\documentclass[12pt]{amsart}
\usepackage{mathrsfs}
\usepackage{latexsym,amssymb,amsmath,amsfonts,graphicx,epstopdf,subfigure,float,enumerate}
\usepackage{pifont}
\usepackage{cite}
\usepackage{multirow}
 \usepackage{color}
\usepackage{todonotes}
\usepackage{xy}
\usepackage{graphicx,subfigure}
\usepackage{tikz}

\xyoption{all}

\def\dist{\mathop{\rm dist}\nolimits}

\newtheorem{theorem}{Theorem}[section]
\newtheorem{thm}[theorem]{Theorem}
\newtheorem{prop}{Proposition}[section]
\newtheorem{cor}[prop]{Corollary}
\newtheorem{lemma}[prop]{Lemma}
\newtheorem{remark}[prop]{Remark}
\newtheorem{defi}[prop]{Definition}

\newtheorem{example}[prop]{Example}
\numberwithin{equation}{section}

\def\R{\mathbb R}

\def\mathscr{\mathcal }

\def\diag{\text{diag}}

\def\bK{{\mathbf K}}
\def\bF{{\mathbf F}}

\newcommand{\ba}{{\mathbf{a}}}
\newcommand{\bb}{{\mathbf{b}}}

\newcommand{\be}{{\mathbf{e}}}

\newcommand{\bi}{{\mathbf{i}}}
\newcommand{\bj}{{\mathbf{j}}}

\newcommand{\bz}{{\mathbf{z}}}

\newcommand{\diam}{\text{diam}}

\begin{document}

\title[Conformal dimension of self-affine sponges]{When the conformal dimension of  a self-affine sponge of Lalley-Gatzouras type is zero}

\author{Yanfang Zhang} \address{School of  Science, Huzhou University,Huzhou, 313000, Zhejiang,China}
\email{03002@zjhu.edu.cn}

\author{Shu-Qin Zhang$^\star$} \address{School of Mathematics and Statistics, Zhengzhou University, Zhengzhou, 450001, Henan, China}
\email{sqzhang@zzu.edu.cn}

\date{\today}
\thanks {The work is supported by NSFS No. 12101566 and Huzhou Natural Science Foundation No. 2022YZ37. }
\thanks{$\star$ Corresponding author.}


\begin{abstract}
 It is well known that if a metric space is uniformly disconnected, then its conformal dimension is zero.
 First,  we  characterize when  a self-affine sponge of Lalley-Gatzouras type is uniformly disconnected.
 Thanks to this characterization,   we show that a self-affine sponge of Lalley-Gatzouras type has conformal dimension
  zero if and only if   it is uniformly disconnected.

\end{abstract}
\maketitle


\section{\bf{Introduction}}
 Quasisymmetry is introduced by Beurling and Ahlfors  \cite{Beurling-Ahlfors-1956} in 1956.
Let $(X, d_X)$ and $(Y,d_Y)$ be two metric spaces. For a given homeomorphism $\eta:[0, \infty)\to [0, \infty)$, a map $f: (X,d_X)\to (Y,d_Y)$ is called  \emph{$\eta$-quasisymmetric} if for all distinct triples $x,y,z\in X$
and $t>0$,
\begin{equation}
\frac{d_X(x,y)}{d_X(y,z)}\leq t \Rightarrow \frac{d_Y(f(x),f(y))}{d_Y(f(y),f(z))}\leq \eta(t).
\end{equation}
We denote by ${\mathcal QS}(X)$ the collection of all quasisymmetric maps defined on $X$.

Pansu \cite{Pansu_1989} introduced \emph{conformal dimension} $\dim_C X$ of $(X,d_X)$, defined as
$$
\dim_C X=\inf_{f\in {\mathcal QS}(X)} \dim_H f(X),
$$
where $\dim_H$ denotes the Hausdorff dimension.
 A metric space $X$ is said to be \emph{minimal for conformal dimension}, or simply \emph{minimal}, if $\dim_C X=\dim_H X$.
 For further background of conformal dimension, we refer to the books  Heinonen \cite{Heinonen01}
  and Mackay and  Tyson \cite{MackayTyson2010}.

Kovalev \cite{Kovalev_2006} proved that if $\dim_H X<1$ then $\dim_C X=0$. So for any metric space $X$,
either $\dim_C X=0$ or $\dim_C X\geq 1$.
When Hausdorff dimension of $X$ is $1$, its conformal dimension can be either $0$ or $1$, see    Tukia \cite{Tukia_1989}
and
Staples and Ward \cite{Staples-Ward-1998}. In \cite{Bishop-Tyson-2001}, Bishop and Tyson constructed minimal Cantor
sets of dimension $\alpha$ for every $\alpha\geq 1$.
In \cite{Hakobyan_2010}, Hakobyan proved that a class of fractal sets in the real line have conformal dimension $1$.

There has been considerable work devoted to the conformal dimension of self-similar sets and self-affine sets.
Tyson and Wu \cite{Tyson-Wu-06} proved that the conformal dimension of Sierpi\'nski gasket is $1$.
Kigami \cite{Kigami-14} gave an upper estimate for the conformal dimension of the Sierpi\'nski carpet.
Bishop and Tyson \cite{Bishop-Tyson-2001PAM} studied the conformal dimension of the antenna set, and Dang and Wen \cite{Dang-Wen-2021} extended the study.   Recently, Binder, Hakobyan and Li \cite{Binder-Hako-Li} show that certain Bedford-McMullen self-affine carpets with uniform fibers are minimal for conformal dimension.

In this paper, we characterize when a diagonal self-affine sponge of Lalley-Gatzouras type has conformal dimension $0$.
Our study is closely related to uniformly disconnectedness.

\begin{defi}[Uniformly disconnected \cite{Heinonen01}]\label{def:uniform}
\emph{A metric space $(X,\rho)$ is uniformly disconnected if there is a constant $\delta_0>0$ such that no $\delta_0$-sequence
exists, where a $\delta_0$-sequence is a sequence of distinct points $(x_0,x_1,\dots,x_n)$ such that  $\rho(x_i,x_{i+1})\leq \delta_0 \rho(x_0,x_n)$.}
\end{defi}

 Heinonen \cite{Heinonen01} proved that

 \begin{lemma}[\cite{Heinonen01}]\label{lem:first}  If a metric space $(X,\rho)$  is uniformly disconnected, then the conformal dimension of $X $ is $0$.
\end{lemma}

It is folklore  that uniformly disconnectedness  is invariant under quasi-symmetric  maps \cite{Heinonen01}. Our main result is as follows.

\begin{thm}\label{thm:dim0}
Let $K$ be a  self-affine sponge of Lalley-Gatzouras type. Then $\dim_C K=0$
if and only if $K$ is uniformly disconnected.
\end{thm}

Thanks to Lemma \ref{lem:first} and together with the result of Kovalev \cite{Kovalev_2006}, we need only to show that if $K$ is not uniformly disconnected, then $\dim_C K\geq 1$.
We will prove this fact in three steps.

Our first step is to present an equivalent   definition of  uniformly disconnectedness.
Let $(X,\rho)$ be a metric space and $E$ be a subset of $X$. For $\delta>0$ and $x,y\in E$, if there exists a sequence $\{x=z_1,\dots,z_n=y\}\subset E$ such that $\rho(z_i,z_{i+1})\le \delta$ holds for $1\le i\le n-1$, then we say $x$ and $y$ are $\delta$-equivalent, and write $x\sim y$. The sequence above is called a \emph{$\delta$-chain} connecting $x$ and $y$. Clearly $\sim$ is an equivalence relation.   $E$ is said to be \emph{$\delta$-connected}, if for any $x,y\in E$,  there is a $\delta$-chain joining $x$ and $y$.
 We call $E$ a \emph{$\delta$-connected component} of $X$, if $E$ is
an equivalence class of the relation $\sim$. (See Rao, Ruan and Yang \cite{RRY_2008}.)

\begin{thm} \label{Thm:Uniform}
Let $(X,\rho)$ be a metric space. Then following two statements are equivalent.

(i) $(X, \rho)$ is uniformly disconnected.

(ii) There is a constant $M_0>0$ such that for any $\delta>0$ and any $\delta$-connected component $U$ of $X$,
  $\diam(U) \leq M_0\delta$.
\end{thm}

\begin{remark}\emph{According to Zhang and Huang \cite{ZH22}, a  metric space $(X,\rho)$ is
called \emph{perfectly disconnected} if the statement (ii) in Theorem holds.
 The above theorem asserts  that   uniform disconnectedness and perfect disconnectedness coincide.
  }
\end{remark}

As the second step, we  characterize when a self-affine sponge of Lalley-Gatzouras type is uniformly disconnected.

The self-affine carpet of Lalley-Gatzouras type was first introduced by Lallay and Gatzouras \cite{Lalley92}, and the self-affine sponges of Lalley-Gatzouras type were studied by \cite{Das16}  which is a very general class of self-affine sponges containing Bedford-McMullen carpets. We will give the  definition  in Section \ref{Affine-Sponge}. For recent works related to self-affine sponges of Lalley-Gatzouras type,
we refer to \cite{Das-Fishman-Simmons19,DouglasHowroyd19, FK23,ZY24}.

\begin{figure}[ht]
\centering
\subfigure{\begin{tikzpicture}[scale=3.5]

\tikzstyle{conefill} = [fill=blue!20,fill opacity=0.8]		
  \tikzstyle{squarefill} = [fill=blue!25,fill opacity=0.9]

 \draw(0,0)node[below]{\small $0$}--(1,0)node[below]{\small $1$}--(1,1)--(0,1)node[left]{\small $1$}--(0,0);
\draw(0,1/5)node[left]{\small $\frac{1}{5}$}--(1,1/5);\draw(0,2/5)node[left]{\small $\frac{2}{5}$}--(1,2/5);\draw(0,3/5)node[left]{\small $\frac{3}{5}$}--(1,3/5);
\draw(0,4/5)node[left]{\small $\frac{4}{5}$}--(1,4/5);\draw(1/3,0)node[below]{\small $\frac{1}{3}$}--(1/3,1);\draw(2/3,0)node[below]{\small $\frac{2}{3}$}--(2/3,1);	
\filldraw[squarefill](1/3,0)--(1/3,1/5,0)--(2/3,1/5)--(2/3,0);
\filldraw[squarefill](0,1/5)--(0,2/5)--(1/3,2/5)--(1/3,1/5);
\filldraw[squarefill](2/3,1/5)--(2/3,2/5)--(1,2/5)--(1,1/5);
\filldraw[squarefill](2/3,3/5)--(2/3,4/5)--(1,4/5)--(1,3/5);
\filldraw[squarefill](1/3,4/5)--(2/3,4/5)--(2/3,1)--(1/3,1);
\filldraw[squarefill](0,3/5)--(0,4/5)--(1/3,4/5)--(1/3,3/5);
\end{tikzpicture}}\hskip 1.5cm		
\subfigure{\begin{tikzpicture}[scale=3.5]

 \tikzstyle{conefill} = [fill=blue!20,fill opacity=0.8]		
  \tikzstyle{squarefill} = [fill=blue!25,fill opacity=0.9]

\draw(0,0)node[below]{\small $0$}--(1,0)node[below]{\small $1$}--(1,1)--(0,1)node[left]{\small $1$}--(0,0);
\draw(1/2,0)node[below]{\small $\frac{1}{2}$}--(1/2,1);  
\draw(1/3,0)node[below]{\small $\frac{1}{3}$}--(1/3,1);		\filldraw[squarefill](0,0)--(0,1/6)node[left]{\small $\frac{1}{6}$}--(1/3,1/6)--(1/3,0)--(0,0);	
\filldraw[squarefill](1/2,0)--(1/2,1/5)--(1,1/5)--(1,0)--(1/2,0);
\filldraw[squarefill](1/2,2/5)--(1/2,3/5)--(1,3/5)--(1,2/5)--(1/2,2/5);
\filldraw[squarefill](1/2,4/5)--(1/2,1)--(1,1)--(1,4/5)--(1/2,4/5);

\filldraw[squarefill](0,1/2)--(1/3,1/2)--(1/3,3/5)--(0,3/5);
\draw(0,0.47)node[left]{\small $\frac{1}{2}$};
\draw(0,0.63)node[left]{\small $\frac{3}{5}$};
\draw(0.25,0.1)node[left]{\small $\phi_1$};	
\draw(0.25,0.56)node[left]{\small $\phi_2$};
\draw(0.83,0.1)node[left]{\small $\phi_3$};
\draw(0.83,0.5)node[left]{\small $\phi_4$};
\draw(0.83,0.9)node[left]{\small $\phi_5$};
\end{tikzpicture}}

\caption{IFS for Bedford-McMullen carpet (left) and self-affine carpet of Lalley-Gatzouras type (right). See Example \ref{Exam:FiberIFS}. }\label{Carpet:Lalley}
\end{figure}
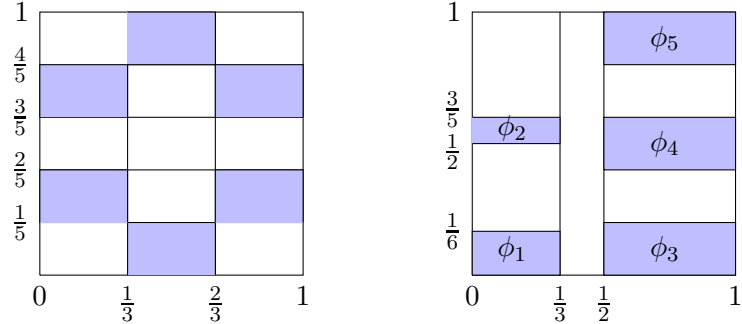

Let $\Phi=\{\phi_j\}_{j=1}^N$ be an IFS of Lalley-Gatzouras type in $\R^d$.  Then $\phi_j$ can be written as
 $$\phi_j(x_1,\dots, x_d)=(\varphi_{j,1}(x_1),\dots,\varphi_{j,d}(x_d)),$$
 where $\varphi_{j,i}(x)$ is a function of the form $ax+b$ with $0<a<1$.
For   precise definition, see Section \ref{Affine-Sponge}.
\emph{For $1\leq \ell \leq d$,} let $\bar{\pi}_\ell: \R^d\rightarrow\R^\ell$ be the projection
$$\bar{\pi}_\ell(x_1,\dots,x_d)=(x_1,\dots, x_\ell),$$
which we call the \emph{$\ell$-th major projection}.
Notice that $\bar{\pi}_d$ is the identity map.
We define the $\ell$-th projection of $\Phi$ to be
 $$
 \Phi_{\{1,\dots, \ell\}}=\{(\varphi_{j,1}, \dots, \varphi_{j,\ell}); 1\leq j\leq N\}, $$
 where if two maps in the right hand side coincide, then we regard it as one element of $\Phi_{\{1,\dots, \ell\}}$.
 Clearly $K_{\ell}=\bar{\pi}_\ell(K)$ is the attractor of $\Phi_{\{1,\dots, \ell\}}$.
 In Section 2, we show that an IFS of Lalley-Gatzouras type can be described by a labelled tree.


\begin{figure}[h] 
\hskip 1.9cm \xymatrix{
&&*++[o][F]{\emptyset}  
\ar[dl]^(0.5){}                               
\ar[dr]^(0.5){}&& \\
& 
*++[o][F]{\frac{x}{3}}
\ar[dl]_(0.6){\frac{x}{6}}\ar[d]^(0.5){\frac{x}{10}+\frac{1}{2}}&
&*++[o][F]{\frac{x+1}{2}}
\ar[dl]_(0.6){\frac{x}{5}}\ar[d]_(0.5){\frac{x+2}{5}}\ar[dr]^(0.58){\frac{x+4}{5}}&&\\
*++[o][F]{\phi_1} &*++[o][F]{\phi_2} &*++[o][F]{\phi_3} &*++[o][F]{\phi_4} &*++[o][F]{\phi_5}
}
\caption{Labeled tree  of the IFS $\Phi$  on the right side of  Figure \ref{Carpet:Lalley}. The fiber IFS related to $\frac{x}{3}$
is $\{\phi_1,\phi_2\}$, and the fiber IFS related to $\frac{x+1}{2}$  is $\{\phi_3, \phi_4,\phi_5\}$.}\label{Labeled Tree}
\end{figure}
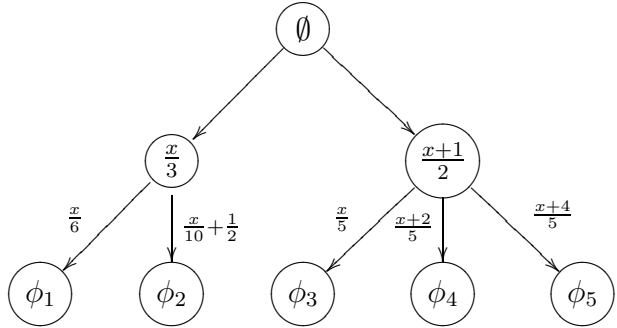


\begin{defi}Let $1\leq \ell \leq d-1$. For $u\in \Phi_{\{1,\dots, \ell\}}$, we define
$$
G(u)=\{h;~(u, h)\in \Phi_{\{1,\dots, \ell+1\}}\},
$$
and  call it a \emph{fiber IFS} of $\Phi$ (or of $K$)  of rank $\ell$.
(For more precise definition, see Section \ref{Affine-Sponge}.)
\end{defi}

Clearly $G(u)$ is an IFS on $[0,1]$.

\begin{thm}\label{thm:Lalley}
Let $K\subset \R^d$ be a  self-affine sponge of Lalley-Gatzouras  type.
Then the following three statements are equivalent.

(i) $K$ is uniformly disconnected.

(ii) All $\bar{\pi}_\ell(K), 1\le \ell \le d,$  are totally disconnected.

(iii) The attractors of all fiber IFS of $K$ are not $[0,1]$.
\end{thm}

\begin{remark}\emph{(1) Huang and Zhang \cite{ZH22}  proved that (i) and (ii)
are equivalent when $K$ is a
self-affine Sierpi\'nski sponge, which is the higher dimensional generalization of Bedford-McMullen carpets.
For precise definition, see Remark \ref{rem:Sierpinski}. However, the argument of \cite{ZH22}  does not work for self-affine sponges of Lalley-Gatzouras type.
 To prove Theorem \ref{thm:Lalley}, we need to employ several new ideas, see  Lemma \ref{lem:product}  and  Lemma \ref{lem:Moran}.}

\emph{(2) Item (iii) provides a very easy algorithm to determine the uniformly disconnectedness of $K$.}
\end{remark}

\begin{figure}[ht]	
\subfigure{\begin{tikzpicture}[scale=4.5]
 \tikzstyle{conefill} = [fill=blue!20,fill opacity=0.8]		
  \tikzstyle{squarefill} = [fill=blue!25,fill opacity=0.9]
\draw(0,0)node[below]{\small $0$}--(1,0)node[below]{\small $1$}--(1,1)--(0,1)node[left]{\small $1$}--(0,0);
\draw(3/4,0)node[below]{\small $\frac{3}{4}$}--(3/4,1);  
\draw(1/4,0)node[below]{\small $\frac{1}{4}$}--(1/4,1);
\draw(5/6,0)node[below]{\small $\frac{5}{6}$}--(5/6,1);

\filldraw[squarefill](0,0)--(0,1/6)node[left]{\small $\frac{1}{6}$}--(1/4,1/6)--(1/4,0)--(0,0);	
\filldraw[squarefill](1/4,1/2)--(3/4,1/2)--(3/4,3/4)--(1/4,3/4)--(1/4,1/2);
\filldraw[squarefill](3/4,1/3)--(5/6,1/3)--(5/6,2/5)--(3/4,2/5)--(3/4,1/3);
\filldraw[squarefill](5/6,8/9)--(1,8/9)--(1,1)--(5/6,1)--(5/6,8/9);

\draw(0,0.52)node[left]{\small $\frac{1}{2}$};
\draw(0,0.73)node[left]{\small $\frac{3}{4}$};
\draw(0.25,0.1)node[left]{\small $\phi_0$};	
\draw(0.57,0.63)node[left]{\small $\phi_1$};
\draw(0.88,0.45)node[left]{\small $\phi_2$};
\draw(0.99,0.95)node[left]{\small $\phi_3$};
\end{tikzpicture}}\\
\hskip 0.5 cm\subfigure{\xymatrix{
&&*++[o][F]{\emptyset}  
\ar[dll]^(0.5){}\ar[dl]^(0.5){}                                 
\ar[dr]^(0.5){}  \ar[drr]^(0.5){}&& \\
*++[o][F]{\frac{x}{4}}
\ar[d]_(0.6){\frac{x}{6}}&*++[o][F]{\frac{2x+1}{4}}\ar[d]_(0.6){\frac{x+2}{4}}
&&*++[o][F]{\frac{x+9}{12}}
\ar[d]_(0.6){\frac{x+5}{15}}&*++[o][F]{\frac{x+5}{6}}\ar[d]_(0.6){\frac{x+8}{9}}&\\
*++[o][F]{\phi_0} &*++[o][F]{\phi_1} &&*++[o][F]{\phi_2} &*++[o][F]{\phi_3}
}}
\caption{An example of self-affine carpet of Lalley-Gatzouras type satisfying the conditions in Theorem \ref{thm:simple-IFS}.}
\end{figure}
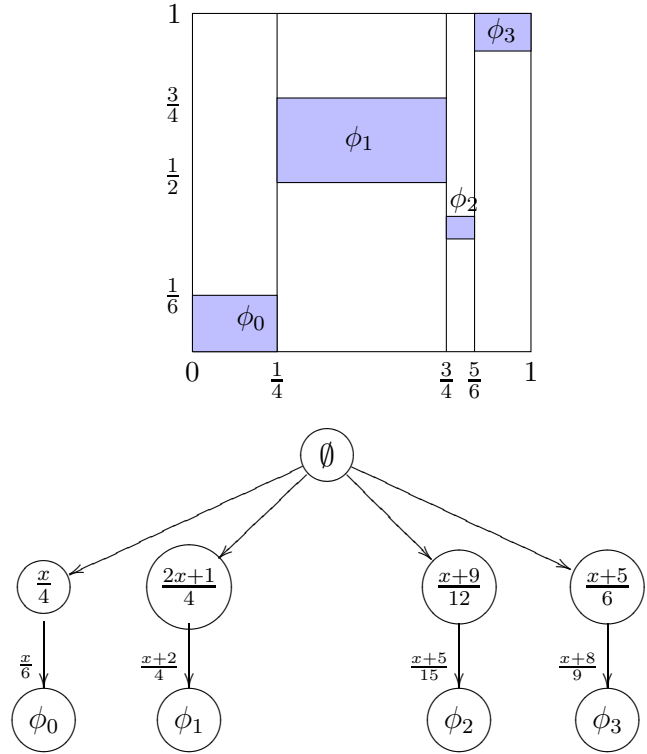

 Our third step is to show that

 \begin{thm}\label{thm:simple-IFS}
 Let $\Phi=\{\phi_j\}_{j=1}^N$ be a diagonal IFS of Lalley-Gatzouras type in ${\mathbb R}^d$ with attractor $\bar{K}$.
 If the fiber IFS of the root $\emptyset$ has cardinality $N$ and with attractor $[0,1]$, and all the other fiber IFS of level $\geq 1$  has cardinality $1$, then $\bar K$ is minimal, more precisely, $\dim_C \bar{K}=\dim_H \bar K= 1$.
 \end{thm}

 To prove the above theorem, first we show that $\bar{K}$ is Lipschitz equivalent to a Cantor set $E$ in ${\mathbb R}$
 with `small gaps'. Then using a result of Hakobyan \cite{Hakobyan_2010}, we show that $\dim_C \bar{K}=1$.
Finally, we show that  if $K$ is a self-affine sponge of Lalley-Gatzouras type and it is not uniformly disconnected, then $K$ contains
a subset $\bar{K}$ such that $\bar K$ satisfies the conditions in Theorem \ref{thm:simple-IFS}, which completes the proof of Theorem \ref{thm:dim0}.

 Finally, as a by-product, we show that

\begin{thm}\label{thm:union} Let $(X,\rho)$ be a metric space and let $E_1,\dots, E_k\subset X$.
If all $E_j$ are uniformly disconnected, then $\bigcup_{j=1}^k E_j$ is also uniformly disconnected.
\end{thm}

The paper is organized as follows. We start with giving some basic notions
related to self-affine sponge of  Lalley-Gatzouras type. Section \ref{Sec:uniform} is contributed
 to prove Theorem \ref{Thm:Uniform}.
 In Section \ref{proof-main-1} we  prove Theorem \ref{thm:Lalley}. In section \ref{proof-main-2}, we prove Theorem \ref{thm:simple-IFS}, which leads to Theorem \ref{thm:dim0}.

\section{\bf{Self-affine sponge of Lalley-Gatzouras type}}\label{Affine-Sponge}
In this section, we first recall the definition of self-affine sponge of Lalley-Gatzouras type,
 then we use a tree with labels to describe the IFS.

\subsection{Self-affine sponge}
 \begin{defi}\emph{
We call $f:\R^d\to \R^d$, $f(x)=Tx+b$ a \emph{diagonal self-affine mapping}
if $T$ is a $d\times d$ diagonal matrix such that all the diagonal entries are positive numbers.
An  IFS
\begin{equation}\label{eq:IFS-diagonal}
\Phi=\{\phi_k(x)=T_kx+b_k\}_{k=1}^N
\end{equation}
 is called a \emph{diagonal self-affine IFS}
if all the maps $\phi_k(x)$ are  diagonal self-affine contractions.
  The attractor of the IFS, that is, the unique nonempty compact set  $K\subset\R^d$  satisfying $K=\cup_{k=1}^N \phi_k(K)$ (\cite{Hut81}),  is called  a \emph{diagonal self-affine sponge}.  Without loss of generality, we will always assume that  $K$ is subset of $d$-dimensional unit cube $ [0,1]^d$.
  (For an alternative definition of diagonal self-affine IFS, see Das and Simmons  \cite{Das16}).}

 \end{defi}

For a diagonal matrix $T_k$, write $T_k=\diag(T_{k,1}, \dots, T_{k,d})$.
The diagonal self-affine IFS $\Phi$ in \eqref{eq:IFS-diagonal} is said to satisfy the \emph{coordinate ordering condition} if
\begin{equation}\label{eq:ordering}
 T_{k,1}>T_{k,2}>\cdots>T_{k,d}, \quad \forall ~k\in \{1,\dots, N\}.
\end{equation}

For $1\leq i \leq d $, let $\bar{\pi}_i: \mathbb{R}^d\rightarrow\mathbb{R}^i$ be the projection   $$\bar{\pi}_i(x_1,\dots,x_d)=(x_1,\dots,x_i).$$
We define the $i$-th major projection IFS of $\Phi$ to be
\begin{equation}\label{Phi_i}
\Phi_{\{1,\dots, i\}}=\left\{
\left (
\begin{array}{c}
x_1 \\
\vdots\\
x_i
\end{array}
\right )\mapsto
\left (
\begin{array}{ccc}
T_{k,1} & &\\
& \ddots &\\
& & T_{k,i}
\end{array}
\right )
\left (
\begin{array}{c}
x_1 \\
\vdots\\
x_i
\end{array}
\right )+\bar{\pi}_i(b_k)
\right\}_{1\leq k\leq N},
\end{equation}
which is a diagonal IFS on $\R^i$. (Here we emphasise that each map occurs at most once in the above IFS.)
Clearly,
$$K_i:=\bar{\pi}_i(K)$$
is the  attractor of the IFS $\Phi_{\{1,\dots, i\}}$.

The following neat projection condition  can be found  in \cite[Definition 3.1]{Das16}.

\begin{defi}\label{def:good}\emph{
Let $K$ be a diagonal self-affine sponge given by $\Phi$ satisfying \eqref{eq:ordering}.
 We say $\Phi$ satisfies the \emph{neat projection condition},
 if for each $i\in \{1,\dots, d\}$, the IFS $\Phi_{\{1,\dots,i\}}$ satisfies
 the open set condition with the open unit cube $\mathbb{I}^i=(0,1)^i$,  that is,
$$
\left\{ f(\mathbb{I}^i);~{f\in \Phi_{\{1,\dots, i\}}} \right\}
$$
are disjoint.}
\end{defi}

\begin{defi}\emph{ A diagonal self-affine sponge is said to be of the  \emph{Lalley-Gatzouras type}, if
it satisfies the coordinate ordering condition
as well as the neat projection condition.}
\end{defi}

\begin{remark}[Self-affine Sierpi\'nski sponge (Keyon and Peres \cite{KP96})]\label{rem:Sierpinski}
\emph{Suppose $2\leq n_1<n_2<\cdots<n_d$ is a sequence of integers. Let $\Phi$ be a diagonal IFS satisfying that}

\emph{(i) $T_k^{-1}\equiv T=\diag(n_1,\dots, n_d)$ for all $k\in \{1,\dots, N\}$;}

\emph{(ii) $b_k\in \prod_{i=1}^d \{0,1,\dots, n_i-1\},$ \\
then we call $\Lambda_\Phi$ a \emph{self-affine Sierpi\'nski sponge}. }
\end{remark}

\subsection{A tree related to IFS of Lalley-Gatzouras type}\label{subsec:tree}
 \begin{defi}\emph{We call $F=\{f_k(x)=a_k x+b_k\}_{k=1}^n$ a \emph{simple IFS} of $[0,1]$, if
$0<a_k<1$ and $\bigcup_{k=1}^n f_k[0,1]\subset [0,1]$  is a  non-overlapping union.}
\end{defi}

 Denote
$$\text{Aff}^+_{[0,1]}=\{x\to a x+b;~ a\in (0,1), b\in [0,1-a]\},$$
which is  a subfamily of affine functions from $[0,1]$ to $[0,1]$.

Let $\Phi$ be an IFS of Lalley-Gatzouras type. We define a \emph{labeled tree}  $\Gamma$ with respect to $\Phi$ as follows.

\begin{itemize}
\item[(1)] \emph{The vertex set of $\Gamma$  is
$$
V=\{\emptyset\}\cup \bigcup_{j=1}^d V_j,
$$
where $\emptyset$ is the root and $V_j=\Phi_{\{1,\dots, j\}}$ for $j=1,\dots, d$.}
\item[(2)]  The edges of $\Gamma$ are given in the following way: All elements of $V_1$ are offsprings of the root $\emptyset$;
for  $f\in V_{j}$ and $g\in V_{j+1}$, we say
$g$ is an \emph{offspring} of $f$ with label $h$ if $g=(f,h)$, where $h\in \text{Aff}^+_{[0,1]}$. Denote the label of $g$ by $L(g)$.
\end{itemize}

Next, we define the fiber IFS of $u\in V$. Let $1\leq j\leq d-1$.  Clearly,  if $u\in V_j$ and $v_1,\dots, v_\ell$ are offsprings of $u$, then the set
$$
G(u)=\{L(v_i); i=1,\dots, \ell\}
$$
constitutes a simple IFS of $[0,1]$. We call $G(u)$ the \emph{fiber IFS} related to $u$, and call it a \emph{fiber IFS} of $K$ of rank $j$.  Specially,    $G(\emptyset)=V_1=\Phi_{\{1\}}$.


\begin{example}\label{Exam:FiberIFS}
\emph{
Let $\Phi=\{\phi_i\}_{i=1}^5$ with
$$
\phi_1(x_1,x_2)=\left (\frac{x_1}{3}, \frac{x_2}{6}\right ), \quad
\phi_2(x_1,x_2)=\left( \frac{x_1}{3}, \frac{x_2}{10}+\frac{1}{2}\right ),
$$
$$\phi_i(x_1,x_2)=\left( \frac{x_1}{2}+\frac{1}{2}, \frac{x_2}{10}+d_i\right ), \quad i\in \{3,4,5\} ,$$
where  $(d_3,d_4,d_5)=(0,2/5,4/5).$
It is easy to check that the attractor of $\Phi$ is a self-affine carpet of Lalley-Gatzouras type. (See
the right picture in Figure \ref{Carpet:Lalley}).   The labeled tree w.r.t. to   $\Phi$
  is illustrated by Figure \ref{Labeled Tree}.  The vertex set of the tree is
$$
V=V_0\cup V_1\cup V_2=\{\emptyset\}\cup\left \{\frac{x}{3},\frac{x+1}{2}\right \}\cup\{\phi_1,\phi_2,\phi_3,\phi_4,\phi_5\}.
$$
Moreover, the fiber IFS related to $\frac{x}{3}$ (of rank $1$) is $\{\frac{x}{6},\frac{x}{10}+\frac{1}{2}\}$ and fiber IFS related to $\frac{x+1}{2}$ (of rank $1$) is $\{\frac{x}{5},\frac{x+2}{5},\frac{x+4}{5}\}$.}
\end{example}

\section{\bf{Characterization of Uniformly Disconnected}}\label{Sec:uniform}

In this section, we  prove Theorem \ref{Thm:Uniform} and Theorem \ref{thm:union}.
 Recall that Theorem \ref{Thm:Uniform} asserts the following two statements   are equivalent.

(i) $(X, \rho)$ is uniformly disconnected.

(ii) There is a constant $M_0>0$ such that for any $\delta$-connected component $U$ of $X$ with $0<\delta<\diam(X)$,
  $\diam(U) \leq M_0\delta$.

\begin{proof}[\textbf{Proof of Theorem \ref{Thm:Uniform}}]
$(i)\Rightarrow(ii)$:
Assume $X$ is uniformly disconnected, let $\delta_0$ be the constant in Definition \ref{def:uniform}
such that there does not exist a $\delta_0$-sequence.

We give a proof by contradiction. Suppose on the contrary that there is a $\delta>0$ and a $\delta$-connected component $C$ of $X$
such that
$$\diam (C)> (2\delta_0^{-1})\delta.$$
Let $a$ and $b$ be two points in $C$ such  that $\rho(a, b)>\diam(C)/2$.
Let $(a=x_0, \dots, x_n=b)$ be a $\delta$-chain joining $a$ and $b$. Then
$$
\rho(x_i, x_{i+1})\leq \delta< \frac{\delta_0~\diam(C)}{2}\leq \delta_0 \rho(x_0, x_n), \quad \text{ for all } 0\leq i\leq n-1,
$$
which contradicts that $X$ is uniformly disconnected.

%

$(ii) \Rightarrow(i)$:
Assume there exists $M_0 > 0$ such that for every $0 < \delta < \diam(X)$, every $\delta$-connected component $U$ of $X$ satisfies $\diam(U) \leq M_0 \delta$.

Let $\delta_0=1/(2M_0)$.
Suppose on the contrary that  $(x_0,\dots, x_n)$ is a $\delta_0$-sequence in $X$.
Denote $\delta=\delta_0 \rho(x_0, x_n)$.
Then  $\rho(x_i, x_{i+1})<\delta_0\rho(x_0, x_n)=\delta$ for all $0\leq i\leq n-1$.
It follows that   $x_0$ and $x_n$ belong to a same $\delta$-connected component of $X$, which we denote by $C$.
Hence
$$\rho(x_0,x_n)\leq \diam(C) \leq M_0\delta=M_0\delta_0 \rho(x_0,x_n)=\rho(x_0,x_n)/2,$$
which is a contradiction.
\end{proof}

\begin{remark}
 \emph{It is  an open problem whether  totally disconnected and uniformly disconnected are equivalent for
  self-similar sets $X$. It is confirmed in two cases: either  $X$ is a self-similar set of finite type ( Xi and Xiong \cite{XX10} showed this for fractal square, but their methods works for self-similar sets of finite type),
 or $X$ is a self-similar set with uniform contraction ratio and satisfying the open set condition (Luo \cite{Luo19}).
}
\end{remark}

In what follows we   build several lemmas we need in next section.
For $E\subset X$, we denote  by $\mathcal{C}_{\delta}(E)$ the collection of the $\delta$-connected component of $E$.
For $E, F\in X$, we denote
$\text{dist}(E, F)=\inf\{\rho(x,y);~x\in E, y\in F\}$.

\begin{lemma}\label{lem:key}  Let $(X,\rho)$ be a metric space and let $E,F\subset X$.  Let $a>0$ and $C\geq 2$  be two constants.
If for every $\delta\geq a$ and every  $U$ in $\mathcal{C}_\delta(E)\cup \mathcal{C}_\delta(F)$, it holds that
\begin{equation}\label{eq:size-1}
\emph{\diam } (U) \leq C\delta,
\end{equation}
 then for every
$\delta\geq a$ and $V\in \mathcal{C}_\delta(E\cup F)$, we have
$$\emph{\diam } (V) \leq 9C^2\delta.$$

\end{lemma}

\begin{proof} Pick $\delta>a$.
For $H\in \mathcal{C}_{5C\delta}(E)$, we set
\begin{equation}\label{eq:MH1}
{\mathcal M}(H)=H\cup \bigcup  \{G;~G\in \mathcal{C}_{\delta}(F) \text{ and } \dist(H,G)\leq  \delta\}.
\end{equation}
The set ${\mathcal M}(H)$ can be regard as the subset of $E\cup F$ `controlled' by $H$. By \eqref{eq:size-1},
it is easy to see that
$$
\diam({\mathcal M}(H))\leq C(5C\delta)+2(\delta+C\delta).
$$

We claim that if $H_1\neq H_2 \in \mathcal{C}_{5C\delta}$, then
\begin{equation}\label{eq:MH2}
\dist({\mathcal M}(H_1), {\mathcal M}(H_2))\geq 2 \delta.
\end{equation}
Pick $x\in {\mathcal M}(H_1)$ and $\epsilon>0$, by \eqref{eq:MH1}, there exists $x'\in H_1$ such that
$ \rho(x,x')\leq C\delta+\delta+\epsilon$.
Similarly, for $y\in {\mathcal M}(H_2)$, there exists $y'\in H_2$ such that $ \rho(y,y')\leq C\delta+\delta+\epsilon$.
Since $ \rho(x',y')>5C\delta$, we have
$$
\rho(x,y)\geq \rho(x',y')-2(C\delta+\delta+\epsilon)\geq 2\delta-2\epsilon.
$$
Let $\epsilon \to 0$, our claim is proved.

Let $V$ be an element of  $\mathcal{C}_{\delta}(E\cup F)$, then by \eqref{eq:MH2} either
 $V\in \mathcal{C}_\delta(F)$ or there exists $H\in \mathcal{C}_{5C\delta}(E)$ such that  $V\subset {\mathcal M}(H)$.
So
$$\diam(V)\leq \max\{\diam({\mathcal M}(H)), C\delta\}\leq C\cdot (5C\delta)+2(\delta+C\delta)\leq 9C^2\delta.$$
The lemma is proved.
\end{proof}

\begin{cor}\label{cor:En} Let $(X,\rho)$ be a metric space and $\{E_i\}_{i=1}^n$ be a sequence of subsets of $ X$. Let $a>0$ and $C\geq 1$ be two constants.
If for every $\delta\geq a$, every $j\in\{1,\dots, n\}$,  and $U\in \mathcal{C}_{\delta}(E_j)$, it holds
$$\emph{\diam}(U)\leq C\delta,$$
 then for every
$\delta\geq a$ and $V\in \mathcal{C}_{\delta}(\bigcup_{j=1}^n E_j)$, we have
$$\emph{\diam}(V)\leq (9C)^{2^{n-1}}/9 ~\delta.$$
\end{cor}

\begin{proof} We prove the corollary by induction on $n$. Suppose that for every
$\delta\geq a$ and every $\delta$-component $V$ of $E^*=\bigcup_{j=1}^{n-1} E_j$, it holds that
$$\diam(V)\leq (9C)^{2^{n-2}}/9~ \delta.$$
By Lemma \ref{lem:key}, we have that for every $\delta\geq a$ and every $V\in \mathcal{C}_{\delta}(E^*\cup E_n)$,
$$\diam(V)\leq 9\cdot [(9C)^{2^{n-2}}/9]^2 \delta=  (9C)^{2^{n-1}}/9~\delta.$$
\end{proof}


Then by Corollary \ref{cor:En} we could give a proof of Theorem \ref{thm:union}.
\begin{proof}[\textbf{Proof of Theorem \ref{thm:union}.}]
It is to show that there exists $M>0$ such that  for any $U \in \mathcal{C}_{\delta}(\bigcup_{j=1}^n E_j)$ it satisfies $\diam(U)\leq M\delta$. Since $E_j$ is uniformly disconnected for $1\leq j \leq n$, there exists $M_j$ such that  for any $U\in \mathcal{C}_{\delta}(E_j)$ with $0<\delta<\diam(X)$, it satisfies $\diam(U)\leq M_j \delta$. Denote $C=\min_{1\leq j \leq n}{M_j}$. Then we have the condition in Corollary  \ref{cor:En}. Set $M=(9C)^{2^{n-1}}/9$, then the theorem is proved.
\end{proof}

\section{\bf{Proof of Theorem \ref{thm:Lalley}}}\label{proof-main-1}

  Let $\Phi=\{\phi_j\}_{j=1}^N$ be an IFS of Lalley-Gatzouras type and $K$ be its attractor.
Define the symbol set $\Sigma=\{1,\dots, N\}
$.

  For any word $\be=e_1\dots e_n\in \Sigma^n$,  let $\phi_{\be}=\phi_{e_1}\circ\dots\circ\phi_{e_n}$.
We refer to
$$
\bK_n=\bigcup_{\be\in \Sigma^n} \phi_{\be}([0,1]^d)
$$
the \emph{$n$-th approximation of the attractor $K$}  which consists of all level-$n$ cylinders.  Let $\pi:\Sigma^\infty\to K$ be the coding map which is given by $\pi(\omega)=\cap_{n\geq 1}\phi_{\omega_{|n}}([0,1]^d),$
where $\omega_{|n}=\omega_1\dots\omega_n$ is the prefix of length $n$. This map assigns each infinite sequence a unique point in the attractor.

 Denote $\phi_j=(\varphi_{j,1}, \dots, \varphi_{j,d})$ and
let $\varphi'_{j,i}$ be the contraction ratio of $\varphi_{j,i}$.
For any cylinder set $\phi_{e_1\dots e_k}(K)$, we define its \emph{width} to be the  length  of its shortest side of $\phi_{e_1\dots e_k}([0,1]^d)$ and denote by $S(\phi_{e_1\dots e_k}(K))$.

\subsection{\textbf{Approximate squares of self-affine sponges}}\label{Appro-Square}
To prove Theorem \ref{thm:Lalley}, the first gradient is a notion called  approximate square,
which is an important  tool to study the variety  properties of $K$.

Let $\delta\in (0,1)$ and $\be=(e_k)_{k\geq 1}\in \Sigma^\infty.$
For each $j\in \{1,\dots, d\}$, let $\ell(j)$ be the smallest integer such that
$$
(\varphi_{e_1,j})'\cdots (\varphi_{e_{\ell(j)},j})'<\delta.
$$
 We call
\begin{equation}
Q(\be, \delta):=\prod_{j=1}^d (\varphi_{e_1,j})\circ\cdots \circ(\varphi_{e_{\ell(j)},j})([0,1])
\end{equation}
the \emph{$\delta$-approximate square} w.r.t. $\be$.


Let ${\bF}=\{f_1,\dots, f_n\}$ be a family of maps on $\R^d$. We shall use the Hutchinson operator
$$
{\bF}(U)=\bigcup_{j=1}^n f_j(U)
$$
for any $U\subset\R^d$.

  Recall that $\Phi_{\{1,\dots, d-1\}}=\{\bar{\pi}_{d-1}(\phi_j);~~1\leq j \leq N\}$. For simplicity,  we denote
$\Phi_{\{1,\dots, d-1\}}=\{\varphi_j\}_{j=1}^{N'}$  and set $\Sigma'=\{1,\dots, N'\}$.

 For $i\in \Sigma'$, let $\bF_i$ be the   fiber IFS corresponding to $\varphi_i$.
For $\bi=i_1\dots i_k\in (\Sigma')^k$, we define
\begin{equation}\label{eq:E-gamma}
E_{\bi}=\bF_{i_1}\circ\cdots \circ \bF_{i_k}([0,1]).
\end{equation}

\begin{lemma}\label{lem:product}
It holds that
$$\Phi^k([0,1]^d)=\bigcup_{\bi\in (\Sigma')^{k}} \varphi_{\bi}([0,1]^{d-1})\times E_{\bi} .$$
\end{lemma}

\begin{proof} We prove the lemma by induction on $k$.
Using the induction hypothesis, we have
$$
\begin{array}{rl}
\Phi^k([0,1]^d) &= \Phi\circ\Phi^{k-1}([0,1]^d) \\
&=\Phi\left (\bigcup_{\bi'\in (\Sigma')^{k-1}} \varphi_{\bi'}([0,1]^{d-1})\times E_{\bi'}\right ) \\
& =\bigcup_{i_1\in \Sigma'}\bigcup_{f\in \bF_{i_1}} (\varphi_{i_1}, f)\left (\bigcup_{\bi'\in (\Sigma')^{k-1}} \varphi_{\bi'}([0,1]^{d-1})\times E_{\bi'}\right )\\
&=\bigcup_{i_1\in \Sigma'} \bigcup_{\bi'\in (\Sigma')^{k-1}} \left (\varphi_{i_1}\circ\varphi_{\bi'}([0,1]^{d-1})\times \bigcup_{f\in \bF_{i_1}}f(E_{\bi'})\right )\\
&=\bigcup_{i_1\in \Sigma'} \bigcup_{\bi'\in (\Sigma')^{k-1}} \left (\varphi_{i_1}\circ\varphi_{\bi'}([0,1]^{d-1})\times E_{i_1\bi'}\right )\\
&=\bigcup_{\bi\in (\Sigma')^{k}} \varphi_{\bi}([0,1]^{d-1})\times E_{\bi}.
\end{array}
$$
The lemma is proved.
\end{proof}

\subsection{\textbf{$\delta$-components of pre-Moran sets}}\label{Delta-compo} To prove Theorem \ref{thm:Lalley},
the second gradient is to analysis the diameter of  $\delta$-components of a class of Cantor sets, called pre-Moran sets.

Let
$$\{\bF_j\}_{j=1}^p$$
be a family of  simple IFS' of $[0,1]$ such that
the attractor of each $\bF_j$ is not $[0,1]$.
Let $\alpha_j$ be the minimal ratio, and $\beta_j$ be the maximal ratio
of contractions in
  $\bF_j$. Denote
$$\alpha_*=\min_{1\leq j\leq p} \alpha_j, \ \ \beta^*=\max_{1\leq j\leq p} \beta_j.$$
We use  $L_j$ to denote the Lebesgue measure of $\bF_j([0,1])$, and use $N_j$ to denote the number of maps in $\bF_j$.
Denote
$$N^*=\max_{1\leq j\leq p} N_j, \  L^*=\max_{1\leq j\leq p} L_j.$$
 Let $g_j$ be the biggest gap of $\bF_j([0,1])\cup\{0,1\}$. Then
$$g_j\geq (1-L^*)/(N^*+2):=g_*.$$

  For $i_1i_2\dots i_k\in\{1,2,\dots,p\}^k$, set
$$E_{i_1\dots i_k}={\bF}_{i_1}\circ \cdots \circ \bF_{i_k}([0,1]),$$
and we call it a \emph{pre-Moran set}. We call $f_{1,h_1}\circ f_{2,h_2}\circ\dots\circ f_{k,h_k}([0,1])$ a \emph{basic interval} of rank $k$ for $f_{j,h_j}\in \bF_{i_j}$.
It is easy to see that $E_{i_1\dots i_k}$ is union of basic intervals of rank $k$.

\begin{remark}\emph{ Let $(e_j)_{j\geq 1}$ be a sequence in $\{1,\dots,p\}^\infty$, then the limit set
$$
\bigcap_{k=1}^\infty {\bF}_{i_1}\circ \cdots \circ \bF_{i_k}([0,1])
$$
is called  a Moran set.  See \cite{H-R-W-W2000,Wen2001}.
}
\end{remark}

Recall that $\mathcal{C}_{\delta}(E)$  denotes all the $\delta$-connected component of  $E\subset \mathbb{R}^d$. The following lemma gives a upper bound of the diameter of $\delta$-connected components of a pre-Moran set.

 \begin{lemma}\label{lem:Moran}
Let $i_1\dots i_k\in \{1,2,\dots,p\}^k$. If
$$\delta\geq \frac{g_*}{\alpha_*}\prod_{j=1}^k \beta_{i_j},$$
then for every  $V\in \mathcal{C}_{\delta}(E_{i_1\dots i_k})$, we have
 $$
\emph{\diam}(V)\leq (2(g_*\alpha_*)^{-1}+1)\delta.
 $$
  \end{lemma}

\begin{proof}
 For $1\leq j\leq k$, denote  ${\bF}_{i_j}=\{f_{j,h}\}_{h=1}^{N_{i_j}}$ and
 denote the contraction ratio of $f_{j,h}$ by $c_{j,h}$.

 Now we regard the sequence $\bF_{i_1}([0,1]), \bF_{i_1}\circ \bF_{i_2}([0,1]), \dots, \bF_{i_1}\circ\dots\circ\bF_{i_k}([0,1])$ as a process of iteration and each of $\bF_{i_1}\circ \dots\circ \bF_{i_j}([0,1])$ consists of basic intervals of rank $j$.
 As soon as a basic interval has length less than $\delta/(g_*\alpha_*)$, then we iterate it one more time and stop iterating.
 Precisely, for each $1\leq q\leq k$, set
 $$
 \Omega^q=\prod_{j=1}^q \{1,\dots, N_{i_j}\}.
 $$
 Define
 $$
 \Omega^q(\delta)=\{u_1\dots u_q\in \Omega^q; ~ c_{1,u_1}\cdots c_{q, u_q}< \delta/(g_*\alpha_*)\leq  c_{1,u_1}\cdots c_{{q-1}, u_{q-1}}  \}.
 $$
 (If $q=1$, by convention, we set the right side of above inequality  to be $1$.)

Denote $I_{{u_1} \dots {u_q}}=f_{1,u_1}\circ \cdots \circ f_{q,u_q}([0,1])$.
Let
\begin{equation}\label{eq:HH}
H=\bigcup_{q=1}^{k-1} \bigcup_{u_1\dots u_q\in \Omega^q(\delta)}\{I_{u_1\dots u_q\ell};~   1\leq \ell\leq N_{i_{q+1}}\}
\end{equation}
 be a union of intervals that continue one more step  beyond those that just fell below the threshold $\delta/(g_*\alpha_*)$.
Since every basic  interval $I$   in $E_{i_1\dots i_k}$ has length smaller than $\prod_{j=1}^k \beta_{i_j}$,
we have
$$
|I|\leq \alpha_*\delta/g_*\leq (\alpha_*)^2 c_{1,u_1}\cdots c_{q, u_{q-1}} \leq  |I_{u_1\dots u_q\ell}|,
$$
which implies that $E_{i_1\dots i_k}$ is a subset of $H$.

 Now we estimate the size of the $\delta$-connected components of $H$.  We first estimate the size of the basic interval. We have
 $$ |I_{u_1\dots u_q}|\geq |c_{1,u_1}\cdots c_{q, u_{q-1}} \alpha_{i_q} |\geq \delta/(g_*\alpha_*)\cdot \alpha_*\geq \delta/g_*,$$
 which implies that
 $$(1-L^*)|I_{u_1\dots u_q}|\geq(N^*+2)\delta,$$
 the above inequality means that the total gap by successors of $I_{u_1\dots u_q}$ is bigger than $(N_{i_{q+1}}+2)\delta$.   Then there is a gap of length bigger than $\delta$, among the intervals $I_{u_1\dots u_q\ell}, 1\leq \ell\leq N_{i_{q+1}}$,
either on the left side of $I_{u_1\dots u_q1}$, or on the right of $I_{u_1\dots u_qN_{i_{q+1}}}$. 
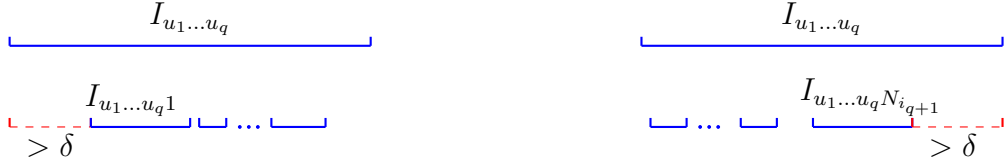
\begin{figure}[h]

\begin{tikzpicture}[scale=1.2]
  \draw[blue, thick] (0,2.5) -- (4,2.5);
  \draw[blue, thick] (0,2.5) -- (0,2.6); %
  \draw[blue, thick] (4,2.5) -- (4,2.6); %
  \node[above] at (2,2.5) {$I_{u_1\dots u_q}$};
  \draw[blue, thick] (7,2.5) -- (11,2.5);
  \draw[blue, thick] (7,2.5) -- (7,2.6); %
  \draw[blue, thick] (11,2.5) -- (11,2.6); %
  \node[above] at (9,2.5) {$I_{u_1\dots u_q}$};

\draw[red,dashed] (0,1.6) -- (0.9,1.6);
  \draw[red, thick] (0,1.6) -- (0,1.7); %
  \draw[red, thick] (0.9,1.6) -- (0.9,1.7); %
\node[] at (0.45,1.4) { $> \delta$};

\draw[blue, thick] (0.9,1.6) -- (2.0,1.6);
  \draw[blue, thick] (0.9,1.6) -- (0.9,1.7); %
  \draw[blue, thick] (2.0,1.6) -- (2.0,1.7); %
\node[above] at (1.35,1.6) {$I_{u_1\dots u_q 1}$};


  \draw[blue, thick] (2.1,1.6) -- (2.4,1.6);
  \draw[blue, thick] (2.1,1.6) -- (2.1,1.7); %
  \draw[blue, thick] (2.4,1.6) -- (2.4,1.7); %
    \fill[blue] (2.55,1.6) circle (0.6pt);
    \fill[blue] (2.65,1.6) circle (0.6pt);
    \fill[blue] (2.75,1.6) circle (0.6pt);

  \draw[blue, thick] (2.9,1.6) -- (3.5,1.6);
  \draw[blue, thick] (2.9,1.6)-- (2.9,1.7); %
  \draw[blue, thick] (3.5,1.6) -- (3.5,1.7); %

\draw[blue, thick] (7.1,1.6) -- (7.5,1.6);
\draw[blue, thick] (7.1,1.6) -- (7.1,1.7); %
\draw[blue, thick] (7.5,1.6) -- (7.5,1.7); %
\fill[blue] (7.65,1.6) circle (0.6pt);
    \fill[blue] (7.75,1.6) circle (0.6pt);
    \fill[blue] (7.85,1.6) circle (0.6pt);

\draw[blue, thick] (8.1,1.6) -- (8.5,1.6);
\draw[blue, thick] (8.1,1.6) -- (8.1,1.7); %
\draw[blue, thick] (8.5,1.6) -- (8.5,1.7); %
\draw[blue, thick] (8.9,1.6) -- (10,1.6);
\draw[blue, thick] (8.9,1.6) -- (8.9,1.7); %
\draw[blue, thick] (10,1.6) -- (10,1.7); %

\node[above] at (9.55,1.6) {$I_{u_1\dots u_q N_{i_{q+1}}}$};

\draw[red,dashed] (10,1.6) -- (11,1.6);
  \draw[red, thick] (10,1.6) -- (10,1.7); %
  \draw[red, thick] (11,1.6) -- (11,1.7); %
\node[] at (10.45,1.4) { $> \delta$};
\end{tikzpicture}
\caption{The red dashed lines show the biggest gap bigger than $\delta$.} 

\end{figure}

So a $\delta$-connected component of $H$ consists of at most two sets in the union
 on the right hand side of \eqref{eq:HH}, and hence the diameter is smaller than
$$2\delta/(g_*\alpha_*)+\delta=(2(g_*\alpha_*)^{-1}+1)\delta.$$
 Therefore,  every $\delta$-connected component of $E_{i_1\dots i_k}$ is smaller than
a corresponding $\delta$-connected component of $H$. The lemma is proved.
 \end{proof}


\subsection{Two more lemmas}
 Let $K\subset \R^d$  be a diagonal self-affine sponge of Lalley-Gatzouras type generated by the IFS $\Phi$.

\begin{lemma}\label{lem:line}  If there exists a fiber IFS of rank $d-1$ with attractor $[0,1]$, then $K$ contains line segment.
\end{lemma}

\begin{proof} Recall that the $(d-1)$-th major projection is
$$\Phi_{\{1,\dots, d-1\}}=\{\varphi_j\}_{j\in \Sigma'}.
$$
 where  $\Sigma'=\{1,2,\dots,N'\}$.
Suppose the fiber IFS of $\varphi_{j_0}$, which we denoted by $\{f_j\}_{j=1}^t$,  possesses attractor $[0,1]$.
Then
$h_j=(\varphi_{j_0}, f_j)\in \Phi$ for $j=1,\dots, t$.
Denote  $H=\{h_j;~1\leq j \leq t\}$, then $H$ is a subset of $\Phi$.
Notice that
$$H([0,1]^d)=\varphi_{j_0}([0,1]^{d-1})\times \bigcup_{j=1}^t f_j([0,1])=\varphi_{j_0}([0,1]^{d-1})\times [0,1].$$
By iterating the above equation, we obtain
$$H^k([0,1]^d)=\varphi^k_{j_0}([0,1]^{d-1})\times [0,1].$$
Denote $\{x_0\}=\bigcap_{k\geq 1} \varphi^k_{j_0}([0,1]^{d-1})$, obviously  $x_0\in \R^{d-1}$, then clearly $\{x_0\}\times [0,1]\subset K$. The lemma is proved.
\end{proof}

\begin{lemma}\label{lem:uni-dis}
If $K_{d-1}=\bar{\pi}_{d-1}(K)$ is uniformly  disconnected and attractors of all  fiber IFS of rank $d-1$ of $K$ are not $[0,1]$, then $K$ is uniformly disconnected.
\end{lemma}
\begin{proof}
Let ${\mathbf Q}_{\delta}$ be the union of all $\delta$-approximate squares of $K_{d-1}$. Set $\epsilon=\sqrt{d}\delta.$
Let us consider the $\epsilon$-connected components of ${\mathbf Q}_{\delta}$.

 For any $U\in \mathcal{C}_{\epsilon}({\mathbf Q}_{\delta})$, we introduce the notation $\# U $ to count the number of $\delta$-approximate squares in ${\mathbf Q}_{\delta}$ contained in $U$.  Denote
\begin{equation*}
M_1=\sup_{\delta>0} \max_{U\in \mathcal{C}_{\epsilon}({\mathbf Q}_\delta)} \#U.
\end{equation*}
Since $K_{d-1}$ is uniformly disconnected, each $\epsilon$-connected component  intersects only finitely many $\delta$-approximate squares.
Therefore $M_1< \infty$.

Pick $W\in \mathcal{C}_{\delta}(K)$.  We will estimate the diameter of $W$ by considering the product of $\delta$-connected components of $K_{d-1}$ and a pre-Moran set.
Notice that  $\pi_{d-1}(W)$ is $\delta$-connected in $K_{d-1}$.
Let $U$ be the element in
$ \mathcal{C}_{\epsilon}({\mathbf Q}_\delta)$ containing $\pi_{d-1}(W)$, then
$W\subset U\times [0,1]$.
 Since $U\in\mathcal{C}_{\epsilon}({\mathbf Q}_\delta)$,  $U$ is the union of $\delta$-approximate squares,\emph{~i.e.}
\begin{equation}\label{equ:U}
U=\bigcup_{j=1}^p Q(\bi^{(j)},\delta),
\end{equation}
where $Q(\bi^{(j)},\delta)$ are $\delta$-approximate squares of $K_{d-1}$  and $\bi^{(j)}\in (\Sigma')^{\infty}$.
For simplicity, denote $Q(\bi^{(j)},\delta)$ by $Q_j$.

Let $t_j$ be the smallest integer such that the width (the smallest side)  of the cylinder of $\varphi_{\bi^{(j)}_1\dots \bi^{(j)}_{t_j}}(K_{d-1}):=T_j$
is smaller than $\delta$, then we have $Q_j\subset T_j$. If $H$ is a cylinder of $K$ of  level $t_j$  such that $\bar\pi_{d-1}(H)=T_j$,
then by the coordinate ordering condition, we have
\begin{equation}\label{eq:width}
\frac{S(H)}{S(T_j)}\to 0,
\end{equation}
 as $\delta \to 0$.

Let $\bF_k$ be the  fiber IFS  related to   $\varphi_k\in \Phi_{\{1,\dots,d-1\}}$   for ${k\in \Sigma'}$.
Then we have a family of simple IFS of $[0,1]$, which we denote by ${\mathcal F}={ \{\bF_k;~k\in \Sigma'\}}$. 

  For the chosen $\bi^{(j)}\in {(\Sigma')^{\infty}}$ in equation \eqref{equ:U}, set
$$
E_j:=E_{\bi^{(j)}_1\dots \bi^{(j)}_{t_j}}= F_{\bi^{(j)}_1}\circ\dots\circ F_{\bi^{(j)}_{t_j}}([0,1]),
$$
to be the pre-Moran set  defined by \eqref{eq:E-gamma}. Then
$$W\subset(U\times [0,1])\cap K \subset \bigcup_{j=1}^p (Q_j\times E_j)\subset  U\times \bigcup_{j=1}^pE_j.$$

 Let $\alpha_*, \beta^*$ and $g_*$ be designed for ${\mathcal F}$ as we did in Subsection \ref{Delta-compo},  here we emphasis that $g_*>0$ since all fiber IFS of rank $d-1$ are not $[0,1]$ by assumption.
By \eqref{eq:width},    there exists $\delta_0>0$ such that if $\delta\leq \delta_0$, then
$$
\prod_{\ell=1}^{t_j}\beta_{\bi^{(j)}_\ell}\leq S(T_j)\alpha_*/ g_*  \leq \delta \alpha_*/ g_*.
$$

Let $\delta\leq \delta'\leq \delta_0$,   then $\delta'\geq\frac{g_*}{\alpha_*}\prod_{\ell=1}^{t_j}\beta_{\bi^{(j)}_\ell}$. By Lemma \ref{lem:Moran},  for  every $E_j=E_{\bi^{(j)}_1\dots \bi^{(j)}_{t_j}}$,  if $V$ is a $\delta'$-connected component of $E_j$,
then
\begin{equation}\label{eq:sizeV}
 \diam(V) \leq  C\delta'
\end{equation}
where $C=2(g_*\alpha_*)^{-1}+1$. If we set
$C=\max\{2(g_*\alpha_*)^{-1}+1, 1/\delta_0\}$, then
\eqref{eq:sizeV} holds for all $\delta'\geq \delta$.

Finally,  by Corollary \ref{cor:En}, if $V$ is a $\delta$-connected component of $\bigcup_{j=1}^pE_j$, then
$$
 \diam(V) \leq (9C^2)^{2^{p-1}}/9\delta:=C_1\delta.
$$
So
$$
\diam(W)\leq \sqrt{\diam(U)^2+C_1^2\delta^2}\leq \delta\sqrt{M_1^2d+C_1^2},
$$
which implies that $K$ is uniformly disconnected.
\end{proof}

\subsection{Proof of Theorem \ref{thm:Lalley}}Now we are in position to prove Theorem \ref{thm:Lalley} which is to show the following three statements are equivalent.

(i) $K$ is uniformly disconnected.

(ii) All $\bar{\pi}_\ell(K), 1\le \ell \le d,$  are totally disconnected.

(iii) The attractors of all fiber IFS of $K$ are not $[0,1]$.

\begin{proof}[\textbf{Proof of Theorem \ref{thm:Lalley}}]
Let $\Phi=\{\phi_i\}_{i\in \Sigma}$ be an IFS of Lalley-Gatzouras type and let $K$ be its attractor.

We prove this theorem by induction on $d$.

Clearly, when $d=1$, the theorem holds.

Now let us assume that the theorem holds for self-affine  sponge of Lalley-Gatzouras type in $\R^{d'}$
with $d'<d$.

(iii)$\Rightarrow$ (i).

By induction hypothesis, we have that $K_{d-1}=\pi_{d-1}(K)$ is  uniformly disconnected. By assumption (iii) that attractors of all fiber IFS of $K$ are not $[0,1]$, then $K$ is uniformly disconnected by Lemma \ref{lem:uni-dis}.

\medskip

(i)$\Rightarrow$(iii). Suppose on the contrary there exists a fiber IFS   possessing attractor $[0,1]$.
That $K$ is uniformly disconnected implies that $K_{d-1}$ is uniformly disconnected, so  by the induction hypothesis,     the attractors of all fiber IFS of $K_{d-1}$ are not   $[0,1]$.
 Therefore,  there exists a fiber IFS of rank $d-1$ possessing attractor $[0,1]$. Now
by Lemma \ref{lem:line},
$K$ contains line segment of the form $\{x_0\}\times[0,1]$ where $x_0\in K_{d-1}$, which is a contradiction.

\medskip

(iii)$\Rightarrow$ (ii).  By induction hypothesis, we have that all $K_{j}$, $j<d$, are totally disconnected.
Since (iii)$\Rightarrow$ (i), so $K_d=K$ is uniformly disconnected, and thus totally disconnected.

\medskip

(ii)$\Rightarrow$(iii). First, by the induction hypothesis, all fiber IFS of rank less than $d-1$ do not
possessing attractor $[0,1]$.
 If there exists a fiber IFS  of rank $d-1$  possessing attractor $[0,1]$, then $K$ contains line segment, which  contradicts
  that $K$ is totally disconnected.


The theorem is proved.
\end{proof}

\section{\bf{Proof of Theorem \ref{thm:dim0}}}\label{proof-main-2}
This section contributes to prove Theorem \ref{thm:simple-IFS}, and thus complete the proof of  Theorem  \ref{thm:dim0}.

\subsection{Binary-tree Cantor set}
Let us first recall the $\{c_i\}$-thick set defined in \cite{Staples-Ward-1998}.  Denote $\Sigma=\{0,1\}$. Let $\Sigma^{0}={\emptyset}$,  $\Sigma^n=\{0,1\}^n$ and    $\Sigma^*=\cup_{n\geq 0}\Sigma^n$. If $\sigma=\sigma_1\dots\sigma_k$ and $\tau=\tau_1\dots\tau_j$ then $\sigma*\tau=\sigma_1\dots\sigma_k\tau_1\dots\tau_j\in \Sigma^{k+j}$.
Let $|\sigma|$ denote the  length of $\sigma\in \Sigma^*$. 

\begin{defi}\rm  Let
$\{J_\sigma;~\sigma\in \Sigma^*\}$
be a family of closed intervals in $\R$ satisfying the following conditions:

(i) $J_\emptyset=[a,b]\subset \R$;

(ii) For each $\sigma\in \Sigma^*$, $J_{\sigma0}$ and $J_{\sigma1}$
are non-overlapping subintervals of $J_\sigma$ such that  $J_{\sigma0}$ is  located on  the left side of $J_{\sigma1}$,
and $J_\sigma \setminus (J_{\sigma0}\cup J_{\sigma1})$ is an open interval
or is the empty set.

We call
$$
E=\bigcap_{n\geq 0}\bigcup_{\sigma\in \Sigma^n} J_\sigma
$$
a \emph{binary-tree Cantor set} if $E$ is totally disconnected.
\end{defi}

 Let $\{c_n\}_{n\geq 1}$ be a sequence of real numbers such that
$0\leq c_n<1$. We say a binary-tree Cantor set $E$ a \emph{$\{c_n\}$-thick set}, if
for any $\sigma\in \Sigma^*$,
$J_\sigma \setminus (J_{\sigma0}\cup J_{\sigma1})$  has length no larger than $c_m |J_{\sigma}|$ where $m=|\sigma|$.

Suppose that  $T\geq 1$ and $E$ is a  $\{c_i\}$-thick set.  $E$ is said to be \emph{$T$-balanced}, if
$$\frac{1}{T}\leq \frac{|J_{\sigma0}|}{|J_{\sigma 1}|}\leq T, \quad \forall \sigma\in \Sigma^*.$$

\begin{lemma}\label{lem:binarytree} Let $E$ be a  binary-tree Cantor set which is $T$-balance. If
$$
\inf_{|\sigma|=n} \frac{\min\{|J_{\sigma 0}|, |J_{\sigma 1}|\}} {|J_\sigma|-|J_{\sigma 0}|-|J_{\sigma 1}|}\to +\infty,
$$
then $\dim_C E=1.$
\end{lemma}

\begin{proof} This is a direct consequence of Theorem 3.2 in Hakobyan \cite{Hakobyan_2010}.
\end{proof}

\subsection{A special IFS of Lalley-Gatzouras type}\label{spe_IFS}
 Let $\Phi=\{\phi_j\}_{j=0}^{m-1}$ be an IFS of Lalley-Gatzouras type
with attractor $\bar{K}$.
  As before, we denote
  $$\phi_j=(\varphi_{j,1}, \dots, \varphi_{j,d}),$$
   and the derivative $\varphi'_{j,i}$ is  the contraction ratio of $\varphi_{j,i}$ and  $\bar{K}_j=\bar{\pi}_j(\bar{K})$ for $j\in \{0,1,\dots,m-1\}$.
Recall that $V= \{\emptyset\}\cup\{V_j; 1\leq j \leq d-1\}$ is the vertex set of the label tree corresponding to $\Phi$.
 Moreover,   for $u\in V$,  the fiber IFS related to $u$ denoted by $G(u)$. see Subsection \ref{subsec:tree}.

We consider a special class of IFS $\Phi=\{\phi_j\}_{j=0}^{m-1}$ of Lalley-Gatzouras type satisfying the following conditions:
\begin{itemize}

 \item[(i)]   the fiber IFS  $G(\emptyset)=V_1$  has attractor $[0,1]$;

 \item[(ii)]  each vertex in  $V_j, 1\leq j\leq d-1$, has only one offspring.
\end{itemize}

Then by (i) the fiber IFS $G(\emptyset)$ has $m$ elements and they are exactly the first coordinate of $\phi_j$, \textit{i.e.},
$G(\emptyset)=(\varphi_{j,1})_{j=0}^{m-1},$
and
\begin{equation}\label{eq:partition}
[0,1]=\bigcup_{j=0}^{m-1}\varphi_{j,1}([0,1])
\end{equation}
is a non-overlapping union, which implies
\begin{equation}\label{eq:ratio}
\sum_{j=0}^{m-1}\varphi^{'}_{j,1}=1.
\end{equation}
From now on, we always assume that
 $\varphi_{j,1}([0,1])$, $0\leq j\leq m-1$, are subintervals of $[0,1]$ from left to right (and $\phi_j$ are also ordered accordingly).

Let $\ba$ be the fixed point of $\phi_{0}$ and $\bb$ be the fixed point of $\phi_{m-1}$. We will regard $\ba$ as the origin
of $\bar{K}$ and regard $\bb$ as the terminus of $\bar{K}$. Moreover,
set
$$\ba_j=\phi_j(\ba), \quad \bb_j=\phi_j(\bb), \quad j=0,1, \dots, m-1.$$
(Clearly $\ba_0=\ba$ and $\bb_{m-1}=\bb$.) We regard $\ba_j$ and $\bb_j$ as origin and terminus of $\bar{K}_j$ respectively.

For  $1\leq j\leq m-1$, denote $\Delta_j={\ba}_j-\bb_{j-1}$, and let $\tau_j$ be the smallest  integer
in $\{1,\dots, d\}$
such that the $(\tau_j)$-th coordinate of $\Delta_j$ is not zero.
If $\Delta_j={\mathbf 0}$, we define $\tau_j=0$ and in this case we  set $\varphi_{\sigma,\tau_j}'=0$ for convention.
  Clearly $\tau_j=0$ or $\tau_j \geq 2$ since the first coordinate of $\Delta_j$ is $0$ by \eqref{eq:partition}.

Let us denote the $j$-th coordinate  projection in $\R^d$ by $\pi_j(x_1,\dots, x_d)=x_j$.   Set $\Omega=\{0,1,\dots,m-1\}$ and $\Omega^*=\cup_{n\geq 0}~\Omega^n$. For $\omega=\omega_1\dots\omega_k\in \Omega^*$, we set $\varphi'_{\omega,j}=\varphi'_{\omega_1,j}\cdot\varphi'_{\omega_2,j}\cdot \dots \cdot \varphi'_{\omega_k,j}$.

\begin{lemma}\label{lem:c0c1}  There exist constants $0<c_0<c_1<\infty$ such that for
$\omega\in \Omega^*\setminus\{\emptyset\}$, it holds that

(i) $c_0\varphi_{\omega,1}'\leq |\phi_\omega(\ba)-\phi_\omega(\bb)|\leq c_1 \varphi_{\omega,1}'.$

(ii) $c_0\varphi_{\omega,\tau_j}' \leq| \phi_\omega(\ba_j)-\phi_\omega(\bb_{j-1})|\leq c_1 \varphi_{\omega,\tau_j}'$ for $1\leq j \leq m-1$.
 \end{lemma}

 \begin{proof} Let    $c_1=d\cdot|\ba-\bb|$ and
 $$c_0=|\pi_1(\ba-\bb)| \wedge  \min\{|\pi_{\tau_j}(\ba_j-\bb_{j-1})|; 1\leq j\leq m-1, \tau_j\geq 2\}.$$
 
 (i) Notice that
 $\phi_\omega(\ba)-\phi_\omega(\bb)=\diag(\varphi_{\omega,1}',\dots, \varphi_{\omega,d}')(\ba-\bb)$, so
$$ |\phi_\omega(\ba)-\phi_\omega(\bb)|\geq \varphi_{\omega,1}'|\pi_1(\ba-\bb)|,$$
so the first inequality in item (i) holds.
 On the other hand,
 $$ |\phi_\omega(\ba)-\phi_\omega(\bb)|\leq \sum_{j=1}^d \varphi_{\omega,j}'|\pi_j(\ba-\bb)|\leq \varphi_{\omega,1}' \cdot d\cdot|\bb-\ba|, $$
 and the second inequality in item (ii) holds.

 (ii) If $\ba_j=\bb_{j-1}$, item (ii) of the lemma  is  true by the convention $\varphi_{\omega,0}'=0$. Suppose $\ba_j\neq \bb_{j-1}$ which means $\tau_j\geq 2$. By the same argument as above,  we have
 $$|\phi_\omega(\ba_j)-\phi_\omega(\bb_{j-1})|\geq \varphi_{\omega,\tau_j}'|\pi_{\tau_j}(\ba_j-\bb_{j-1})|\geq  \varphi_{\omega,\tau_j}' c_0 ,
 $$
 and
 $$ |\phi_\omega(\ba_j)-\phi_\omega(\bb_{j-1})|\leq \sum_{i=\tau_j}^d \varphi_{\omega,i}'|\pi_i(\ba_j-\bb_{j-1})|\leq d\cdot\varphi_{\omega,\tau_j}' \cdot |\ba_{j}-\bb_{j-1}|\leq \varphi_{\omega,\tau_j}'c_1.$$
 The lemma is proved.
 \end{proof}

\subsection{An $m$-tree Cantor set}

   Now we construct a  $m$-tree Cantor  set $E$ tailored  for the above special IFS $\Phi$ such that $E$ is bi-Litchis equivalent to $\bar{ K}$.

   Let
\begin{equation}\label{eq:L}
L=1+\sum_{\beta\in \Omega^*} \sum_{j=1}^{m-1}  \varphi_{\beta,\tau_j}',
\end{equation}
 where $\tau_j$ is defined in Subsection \ref{spe_IFS},
  we set $\varphi_{\emptyset,\tau_j}'=\Delta_j$ for convention.

We construct a $m$-branch-tree Cantor set as follows.

 (i) \textbf{Initial interval}: we   set the initial interval to be $J_\emptyset=[0, L]$;

 (ii) \textbf{Nested structure}:   For $\omega\in \Omega^*$,  $J_{\omega j}$, $j=0,1,\dots, m-1$, are non-overlapping  closed sub-intervals of $J_\omega$ located from left to right.
 Moreover, we require the left end point of $J_{\omega 0}$ coincide with that of  $J_\omega$
   and the right end point of $J_{\omega (m-1)}$ coincide with that of  $J_\omega$.

 (iii) \textbf{Length of cylinders}: For $\omega\in \Omega^*$, we set
 \begin{equation}\label{eq:cylinder-length}
 |J_\omega|=\varphi_{\omega,1}'+\sum_{\beta\in \Omega^*} \sum_{j=1}^{m-1}  \varphi_{\omega\beta,\tau_j}'.
 \end{equation}
 (We remark that the above formula holds for $\omega=\emptyset$  if we set $\varphi'_{\emptyset,1}=1$.)

 (iv) \textbf{Placements of cylinders:}  For $j\in \{1,\dots,m-1\}$ and $\omega\in \Omega^*$, we set the gap between $J_{\omega(j-1)}$ and $J_{\omega j}$ to be
 $ \varphi_{\omega,\tau_j}':=g_{\omega,j}$.

 \begin{lemma} It holds that
 $$
 |J_\omega|=\sum_{j=0}^{m-1}|J_{\omega j}|+\sum_{j=1}^{m-1}  g_{\omega,j}.
 $$
 \end{lemma}

 \begin{proof}
 By \eqref{eq:cylinder-length}, we have
 $$\sum_{i=0}^{m-1}|J_{\omega i}|=\sum_{i=0}^{m-1}\varphi'_{\omega i,1}+\sum_{i=0}^{m-1}\sum_{\beta\in \Omega^*} \sum_{j=1}^{m-1}  \varphi_{\omega i \beta,\tau_j}'.$$
 Since $\sum_{j=0}^{m-1}\varphi^{'}_{j,1}=1$, we have $\sum_{i=0}^{m-1}\varphi'_{\omega i,1}=\varphi'_{\omega,1}\sum_{i=0}^{m-1}\varphi'_{i,1}=\varphi'_{\omega,1}$.
 Thus we have
\begin{align*}
\sum_{i=0}^{m-1}|J_{\omega i}|+\sum_{j=1}^{m-1}  g_{\omega,j}&=\varphi'_{\omega,1}+\sum_{i=0}^{m-1}\sum_{\beta\in\Omega^*} \sum_{j=1}^{m-1}  \varphi_{\omega i \beta,\tau_j}'+\sum_{j=1}^{m-1}\varphi'_{\omega ,\tau_j}\\
&=\varphi'_{\omega,1}+\sum_{\beta'\in \Omega^*} \sum_{j=1}^{m-1}  \varphi_{\omega \beta',\tau_j}'=|J_{\omega}|.
\end{align*}

 \end{proof}

Define
\begin{equation}\label{eq:Moran}
E=\bigcap_{n\geq 0} \bigcup_{|\omega|=n} J_\omega,
\end{equation}
and we call it an \emph{$m$-tree Cantor set}.

\begin{lemma}\label{lem:J-sigma}  For $L$ given in \eqref{eq:L}, we have
$$
|J_\omega|\leq L\varphi_{\omega,1}',\quad \forall ~\omega\in \Omega^*\setminus\{\emptyset\}.
$$
\end{lemma}

\begin{proof}
Here we recall that $\{\phi_j\}_{j=0}^{m-1}$ is an IFS of Lalley type then we have
$$\varphi'_{\omega,1}\geq \varphi'_{\omega,2}\geq\dots\geq \varphi'_{\omega,d}\quad \forall ~\omega\in \Omega^*\setminus\{\emptyset\}.$$
 Therefore,
\begin{align*}
|J_{\omega}|=&\varphi_{\omega,1}'+ \sum_{\beta\in \Omega^*} \sum_{j=1}^{m-1}  \varphi_{\omega \beta,\tau_j}'
 = \varphi_{\omega,1}'+ \sum_{\beta\in \Omega^*} \sum_{j=1}^{m-1} \varphi'_{\omega ,\tau_j} \varphi_{  \beta,\tau_j}'\\
 \leq  & \varphi_{\omega,1}'+ \varphi'_{\omega ,1} \sum_{\beta\in \Omega^*} \sum_{j=1}^{m-1}  \varphi_{ \beta,\tau_j}'
  = \varphi_{\omega,1}'+(L-1)\varphi'_{\omega,1}=\varphi'_{\omega,1}L.
 \end{align*}
The lemma is proved.
\end{proof}

\begin{thm}\label{thm:Lip} The self-affine sponge $\bar{K}$ introduced in Section \ref{spe_IFS} and $m$-tree Cantor set $E$ in \eqref{eq:Moran} are Lipschitz equivalent.
\end{thm}

\begin{proof} Let $\pi_{\bar{K}}:\Omega^\infty\to \bar{K}$ be the coding projection of $\bar{K}$, and let
$\pi_E:\Omega^\infty \to E$ be the coding projection of $E$.

First, we claim that  $\pi_{\bar{K}}(\bi)=\pi_{\bar{K}}(\bj)$ if and only if $\pi_E(\bi)=\pi_E(\bj)$.
Indeed, $\pi_{\bar{K}}(\bi)=\pi_{\bar{K}}(\bj)$ implies that
\begin{equation}\label{eq:bifurcation}
\bi=i_1\dots i_k i (m-1)^\infty \quad \text{ and } \quad \bj=i_1\dots i_k(i+1)0^\infty.
\end{equation}
 Moreover, this further means
that
\begin{equation}\label{eq:i+1}
\tau_{i+1}=0.
\end{equation} Clearly, \eqref{eq:bifurcation} together with \eqref{eq:i+1} also imply that $\pi_{\bar{K}}(\bi)=\pi_{\bar{K}}(\bj)$.
Similarly, one can show that  $\pi_E(\bi)=\pi_E(\bj)$ if and only if \eqref{eq:bifurcation} and \eqref{eq:i+1} hold. Our claim is proved.

Now we define $f: \bar{K}\to E$ as
$$f(x)=\pi_E(\bi),$$
where $\bi\in \pi_{\bar{K}}^{-1}(x)$. The claim above justifies  that $f$ is well defined.

We are going to show that there exists $C_0>1$ such that
 \begin{equation}\label{eq:Lip}
 C_0^{-1}|x-y|\leq |f(x)-f(y)|\leq C_0|x-y|,\quad \forall~ x,y\in \bar{K}.
 \end{equation}

Since both
$ \{\phi_\omega(\ba);~\omega\in \Omega^*\}$ and $ \{\phi_\omega(\bb);~\omega\in \Omega^*\}$
are dense in $\bar{K}$, and that both $\{\phi_\omega(\ba);~\omega\in \Omega^n \}_{n\geq 1}$
and $\{\phi_\omega(\bb);~\omega\in \Omega^n \}_{n\geq 1}$ are increasing, we only need to show that \eqref{eq:Lip} holds for all $n\geq 1$ and
$x=\phi_\alpha(\ba)$ and $y=\phi_\beta(\bb)$
 with $\alpha, \beta\in \Omega^n$ and $\alpha\prec \beta$.

Denote $u=f(x)$ and $v=f(y)$.  By the assumption of $x$ and $y$, we know that $x$ has coding $\alpha 0^{\infty}$ and $y$ has coding $\beta (m-1)^{\infty}$. 
 Let $\bz=\alpha\wedge \beta$ and denote $p=|\bz|$.

We first prove that
 \begin{equation}\label{eq:geq}
 c_1|f(x)-f(y)|\geq |x-y|,\quad \forall ~x,y\in \bar{K}. 
 \end{equation}
 where $c_1=d\cdot |\ba-\bb|$  is the constant in Lemma \ref{lem:c0c1}.
 By Lemma \ref{lem:c0c1} (i), we have
 \begin{equation}\label{eq:Var1}
 |\phi_\omega(\bb)-\phi_\omega(\ba)|\leq c_1\varphi_{\omega,1}'\leq c_1|J_\omega|,\quad \forall~ \omega\in \Omega^*\setminus\{\emptyset\}.
 \end{equation}
 By Lemma \ref{lem:c0c1} (ii),  $\forall~ \omega\in \Omega^*\setminus\{\emptyset\}, j=1,\dots, m-1$, we have
 \begin{equation}\label{eq:Var2}
 |\phi_\omega(\ba_j)-\phi_\omega(\bb_{j-1})| \leq c_1\varphi'_{\omega,\tau_j}=c_1g_{\omega,j}.
 \end{equation}
 Now let $x=\phi_\alpha(\ba)$ and $y=\phi_\beta(\bb)$. Then $x$ and $y$ can be connected by a broken line with end points
 in $ \bigcup_{n\geq 1}\bigcup_{|\omega|=n}\phi_\omega(\{\ba,\bb\})$. Hence, by \eqref{eq:Var1} and  \eqref{eq:Var2}, we obtain
 $$
 |x-y|\leq c_1|f(x)-f(y)|=c_1|u-v|.
 $$
 
 Next, we prove the other direction inequality of \eqref{eq:Lip}, i.e.,
 $$|u-v|\leq C_0|x-y|, \text{where } x=\phi_{\alpha}(\ba),y=\phi_{\beta}(\bb) \text{ with } \alpha,\beta\in \Omega^n. $$

  Denote $\alpha=\alpha_1\dots\alpha_n\dots$ and $\beta=\beta_1\dots\beta_n\dots$. We know that $\alpha_1\dots\alpha_p=\beta_1\dots\beta_p$.
  Here we assume that $f(x)=u< v=f(y)$.  We divide the proof in the following three cases.
 \medskip

 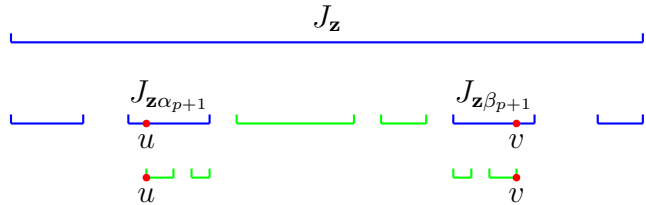
\begin{figure}[h]

\begin{tikzpicture}[scale=1.2]
  \draw[blue, thick] (1,2.5) -- (8,2.5);
  \draw[blue, thick] (1,2.5) -- (1,2.6); %
  \draw[blue, thick] (8,2.5) -- (8,2.6); %
  \node[above] at (4.5,2.5) {$J_{\bz}$};
\draw[blue, thick] (1,1.6) -- (1.8,1.6);
  \draw[blue, thick] (1,1.6) -- (1,1.7); %
  \draw[blue, thick] (1.8,1.6) -- (1.8,1.7); %

  \draw[blue, thick] (2.3,1.6) -- (3.2,1.6);
  \draw[blue, thick] (2.3,1.6) -- (2.3,1.7); %
  \draw[blue, thick] (3.2,1.6) -- (3.2,1.7); %
  \node[above] at (2.75,1.6) {$J_{{\bz}\alpha_{p+1}}$};
  \fill[red] (2.5,1.6) circle (1.2pt); %
  \node[below] at (2.5,1.6) {$u$};

  \draw[green, thick] (3.5,1.6) -- (4.8,1.6);
  \draw[green, thick] (3.5,1.6) -- (3.5,1.7); %
  \draw[green, thick] (4.8,1.6) -- (4.8,1.7); %

  \draw[green, thick] (5.1,1.6) -- (5.6,1.6);
  \draw[green, thick] (5.1,1.6) -- (5.1,1.7); %
  \draw[green, thick] (5.6,1.6) -- (5.6,1.7); %

  \draw[blue, thick] (5.9,1.6) -- (6.8,1.6);
  \draw[blue, thick] (5.9,1.6)-- (5.9,1.7); %
  \draw[blue, thick] (6.8,1.6) -- (6.8,1.7); %
  \node[above] at (6.35,1.6) {$J_{{\bz}\beta_{p+1}}$};
  \fill[red] (6.6,1.6) circle (1.2pt); 
  \node[below] at (6.6,1.6) {$v$};
\draw[blue, thick] (7.5,1.6) -- (8,1.6);
\draw[blue, thick] (7.5,1.6) -- (7.5,1.7); %
\draw[blue, thick] (8,1.6) -- (8,1.7); %

\draw[green, thick] (2.5,1) -- (2.8,1);
\draw[green, thick] (2.5,1) -- (2.5,1.1); %
\draw[green, thick] (2.8,1) -- (2.8,1.1); %
  \fill[red] (2.5,1) circle (1.2pt); 
  \node[below] at (2.5,1) {$u$};
\draw[green, thick] (3.0,1) -- (3.2,1);
\draw[green, thick] (3.0,1) -- (3.0,1.1); %
\draw[green, thick] (3.2,1) -- (3.2,1.1); %
\draw[green, thick] (5.9,1) -- (6.1,1);
  \draw[green, thick] (5.9,1)-- (5.9,1.1); %
  \draw[green, thick] (6.1,1) -- (6.1,1.1); %
\draw[green, thick] (6.3,1) -- (6.6,1);
  \draw[green, thick] (6.3,1)-- (6.3,1.1); %
  \draw[green, thick] (6.6,1) -- (6.6,1.1); %
  \fill[red] (6.6,1) circle (1.2pt); 
  \node[below] at (6.6,1) {$v$};

\end{tikzpicture}
\caption{$\beta_{p+1} \geq \alpha_{p+1} + 2$. The green intervals are elements in ${\mathcal B}([u,v])$.} \label{Thm5.1:case1}

\end{figure}

\textit{Case 1.} $\beta_{p+1}\geq \alpha_{p+1}+2$.

  Let
$${\mathcal B}([u,v])=\{J_\omega; ~ J_\omega\subset [u,v] \text{ but } J_{\omega^-}\not\in [u,v]\}, $$
where $\omega^-$ denotes the prefix of $\omega$ obtained by deleting the last letter of $\omega$.
Let $r_*=\min\{\varphi_{j,1};~ j=0,\dots, m-1\}$.

 Before preceding the proof, we introduce a notion of bank. Let $G$ be the gap between cylinders $J_{\sigma i}$ and $J_{\sigma (i+1)}$, we call $J_{\sigma i}$ the \emph{left bank} of $G$ and $J_{\sigma (i+1)}$ the \emph{right bank} of $G$. 
 
Let $G$ be a gap in $[u,v]$ between two elements in ${\mathcal B}([u,v])$.
The crucial fact is that at least one of the banks of $G$ belongs to ${\mathcal B}([u,v])$ by the assumption of Case 1. (See Figure \ref{Thm5.1:case1}.)
Let $J_{\alpha}$ a bank of $G$ belonging to ${\mathcal B}([u,v])$.
   Then
$$|G|\leq |J_{\alpha^-}|\leq L\varphi'_{\alpha^-,1}\leq \frac{L\varphi'_{\alpha,1}}{r_*}\leq  \frac{L}{r_*}|J_\alpha|:=C'|J_\alpha|,$$
where the second inequality is due to Lemma \ref{lem:J-sigma}, and the last inequality is due to \eqref{eq:cylinder-length}.

Since  an element  in ${\mathcal B}([u,v])$ is adjacent to at most two gaps, we have
\begin{equation}\label{eq:Case=1}
\begin{array}{rl}
|u-v| &\leq (1+2C') \sum_{ J_\alpha \in {\mathcal B}([u,v])]} |J_\alpha|\\
&\leq  L(1+2C')\sum_{ J_\alpha \in {\mathcal B}([u,v])]} \varphi_{\alpha,1}'\\
&\leq L(1+2C') |\pi_1(x)-\pi_1(y)|.
\end{array}
\end{equation}
It follows that
$$
|u-v|\leq  L(1+2C')|x-y|.
$$

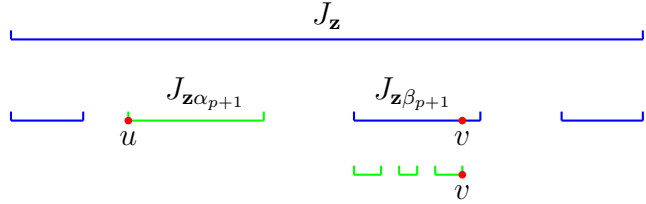
\begin{figure}[h]

\begin{tikzpicture}[scale=1.2]
  \draw[blue, thick] (1,2.5) -- (8,2.5);
  \draw[blue, thick] (1,2.5) -- (1,2.6); %
  \draw[blue, thick] (8,2.5) -- (8,2.6); %
  \node[above] at (4.5,2.5) {$J_{\bz}$};
\draw[blue, thick] (1,1.6) -- (1.8,1.6);
  \draw[blue, thick] (1,1.6) -- (1,1.7); %
  \draw[blue, thick] (1.8,1.6) -- (1.8,1.7); %

  \draw[green, thick] (2.3,1.6) -- (3.8,1.6);
  \draw[green, thick] (2.3,1.6) -- (2.3,1.7); %
  \draw[green, thick] (3.8,1.6) -- (3.8,1.7); %
  \node[above] at (3.15,1.6) {$J_{{\bz}\alpha_{p+1}}$};
  \fill[red] (2.3,1.6) circle (1.2pt); 
  \node[below] at (2.3,1.6) {$u$};

  \draw[blue, thick] (4.8,1.6) -- (6.2,1.6);
  \draw[blue, thick] (4.8,1.6)-- (4.8,1.7); %
  \draw[blue, thick] (6.2,1.6) -- (6.2,1.7); %
  \node[above] at (5.45,1.6) {$J_{{\bz}\beta_{p+1}}$};
  \fill[red] (6.0,1.6) circle (1.2pt); 
  \node[below] at (6.0,1.6) {$v$};
\draw[blue, thick] (7.1,1.6) -- (8,1.6);
\draw[blue, thick] (7.1,1.6) -- (7.1,1.7); %
\draw[blue, thick] (8,1.6) -- (8,1.7); %

\draw[green, thick] (4.8,1) -- (5.1,1);
  \draw[green, thick] (4.8,1)-- (4.8,1.1); %
  \draw[green, thick] (5.1,1) -- (5.1,1.1); %
\draw[green, thick] (5.3,1) -- (5.5,1);
  \draw[green, thick] (5.3,1)-- (5.3,1.1); %
  \draw[green, thick] (5.5,1) -- (5.5,1.1); %
  \draw[green, thick] (5.7,1) -- (6.0,1);
  \draw[green, thick] (5.7,1)-- (5.7,1.1); %
  \draw[green, thick] (6.0,1) -- (6.0,1.1); %
  \fill[red] (6.0,1) circle (1.2pt); 
  \node[below] at (6.0,1) {$v$};

\end{tikzpicture}
\caption{$\beta_{p+1} = \alpha_{p+1} + 1$ and $x$ is the origin of a cylinder of order $p+1$. The green intervals are in ${\mathcal B}([u,v])$.} \label{Thm5.1:case2}

\end{figure}

\textit{Case 2.}   $\beta_{p+1}= \alpha_{p+1}+1$ and $x$ is the origin of a cylinder of order $p+1$,
or $y$ is the  terminus of a cylinder of order $p+1$.

In this case,   for any gap $G$ between two adjacent cylinders in ${\mathcal B}([u,v])$, still at least one of  its banks
belongs to ${\mathcal B}([u,v])$. Therefore,
all the relations in \eqref{eq:Case=1} still holds.

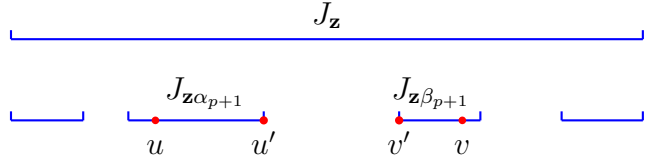
\begin{figure}[h]

\begin{tikzpicture}[scale=1.2]
  \draw[blue, thick] (1,2.5) -- (8,2.5);
  \draw[blue, thick] (1,2.5) -- (1,2.6); %
  \draw[blue, thick] (8,2.5) -- (8,2.6); %
  \node[above] at (4.5,2.5) {$J_{\bz}$};
\draw[blue, thick] (1,1.6) -- (1.8,1.6);
  \draw[blue, thick] (1,1.6) -- (1,1.7); %
  \draw[blue, thick] (1.8,1.6) -- (1.8,1.7); %

  \draw[blue, thick] (2.3,1.6) -- (3.8,1.6);
  \draw[blue, thick] (2.3,1.6) -- (2.3,1.7); %
  \draw[blue, thick] (3.8,1.6) -- (3.8,1.7); %
  \node[above] at (3.15,1.6) {$J_{{\bz}\alpha_{p+1}}$};
  \fill[red] (2.6,1.6) circle (1.2pt); 
  \node[below] at (2.6,1.5) {$u$};
   \fill[red] (3.8,1.6) circle (1.3pt); 
  \node[below] at (3.8,1.6) {$u'$};

  \draw[blue, thick] (5.3,1.6) -- (6.2,1.6);
  \draw[blue, thick] (5.3,1.6)-- (5.3,1.7); %
  \draw[blue, thick] (6.2,1.6) -- (6.2,1.7); %
  \node[above] at (5.65,1.6) {$J_{{\bz}\beta_{p+1}}$};
  \fill[red] (6.0,1.6) circle (1.2pt); 
  \node[below] at (6.0,1.5) {$v$};
    \fill[red] (5.3,1.6) circle (1.3pt); 
  \node[below] at (5.3,1.6) {$v'$};
\draw[blue, thick] (7.1,1.6) -- (8,1.6);
\draw[blue, thick] (7.1,1.6) -- (7.1,1.7); %
\draw[blue, thick] (8,1.6) -- (8,1.7); %

\end{tikzpicture}
\caption{$\beta_{p+1} = \alpha_{p+1} + 1$} \label{Thm5.1:case3}

\end{figure}

\textit{Case 3.}   $\beta_{p+1}= \alpha_{p+1}+1$.

Let $u'$ be the terminus of the cylinder $J_{\bz (\alpha_{p+1})}$ and let $v'$ be the origin
of the cylinder  $J_{\bz (\alpha_{p+1}+1)}$.(See Figure \ref{Thm5.1:case3}.)
Denote $t=\alpha_{p+1}+1$.  Then the gap between $J_{\bz (\alpha_{p+1})}$ and $J_{\bz (\alpha_{p+1}+1)}$ is $g_{\bz, t}=[u',v']$ and
\begin{equation}\label{eq:case2}
 |u-v|= |u-u'|+g_{\bz, t}+ |v'-v|. 
\end{equation}
Denote $x'=f^{-1}(u')$ and $y'=f^{-1}(v')$.  Clearly $x'$ is the terminus of a cylinder of rank $p+1$ and $y'$ is the origin of a cylinder of rank $p+1$. 

 Now we consider the pair $x$ and $x'$.
 Let $  \theta$ and $\theta'$ be a coding of $x$ and $x'$, respectively. Then
 $|\theta\wedge\theta'|\geq p$.
Therefore, $x$ and $x'$ satisfy the assumption of Case 2, so by \eqref{eq:Case=1}, we have
$$
|u-u'|\leq L(1+2C')|\pi_1(x)-\pi_1(x')|.
$$
Similarly,
$$
|v-v'|\leq L(1+2C')|\pi_1(y)-\pi_1(y')|.
$$
Then
$$
\begin{array}{rl}
|x-y| & \geq |x'-y'|-|x-x'|-|y-y'|\\
&\geq |x'-y'|-c_1|u-u'|-c_1|v-v'|\\
&\geq |x'-y'|-c_1L(1+2C')|\pi_1(x)-\pi_1(x')|-c_1L(1+2C')|\pi_1(y)-\pi_1(y')|\\
& \geq c_0\varphi_{\bz, \tau(t)}'-2c_1L(1+2C')|x-y|,
\end{array}
$$
where the second inequality is from \eqref{eq:geq} and  $c_0$ is the constant from Lemma \ref{lem:c0c1}.
It follows that
$$
g_{\bz, t}=\varphi_{\bz, \tau(t)}'\leq \frac{(1+2c_1L(1+2C'))}{c_0}|x-y|.
$$
This together with \eqref{eq:case2} imply
$$
|u-v|\leq \left ( \frac{(1+2c_1L(1+2C'))}{c_0}+2L(1+2C')\right )|x-y| .
$$
Take $C_0=\max\{c_1,\frac{1+2c_1L(1+2C')+2c_0L(1+2C')}{c_0}\}$, then \eqref{eq:Lip} is satisfied.
The theorem is proved.
\end{proof}

\subsection{Conformal dimension}

\begin{lemma}\label{lem:binary} The $m$-tree Cantor set $E$ has conformal dimension $1$.
\end{lemma}
\begin{proof}
 According to $E$, we construct a binary-tree Cantor set $F$ in the following way.
Set  $F_\emptyset=[0, L]$.

Now we define the interval $F_\sigma$, $\sigma\in \{0,1\}^*$, inductively.

Suppose $F_\sigma$ has been defined in the way such that  either $F_\sigma=J_\alpha$ or $F_\sigma$ is a finite union of consecutive subintervals of $J_\alpha$
for some $\alpha\in \Sigma^*$, say,
$$
F_\sigma=\bigcup_{j=k_1}^{k_2} J_{\alpha j}.
$$
Here $0\leq k_1\leq k_2\leq m-1$, but $(k_1,k_2)\neq (0, m-1)$.

If  $F_\sigma=J_\alpha$ ,  we define
$$
F_{\sigma 0}= J_{\alpha 0},\quad F_{\sigma 1}=\bigcup_{j=1}^{m-1} J_{\alpha j}.
$$
If $k_1<k_2$, we define
$$
F_{\sigma 0}= J_{\alpha (k_1)},\quad F_{\sigma 1}=\bigcup_{j=k_1+1}^{k_2} J_{\alpha j}.
$$
By Lemma \ref{lem:J-sigma} and inequality \eqref{eq:cylinder-length}, we have
$$|J_{\alpha}|\leq L\varphi_{\alpha,1}'\leq L\varphi_{\alpha j,1}'/r_*\leq L/r_* |J_{\alpha j}|,$$
 which implies that
$$
\frac{r_*}{L}\leq \frac{|F_{\sigma 0}|}{|F_{\sigma 1}|}\leq \frac{L}{r_*},
$$
so $F$ is $T$-balanced with $T=L/r_*$.

Next, if $\text{dist}(F_{\sigma 0}, F_{\sigma 1})>0$, we have
$$
\frac{|F_{\sigma 0}|}{\text{dist}(F_{\sigma 0}, F_{\sigma 1})}\geq \frac{|J_{\alpha k_1 0}|}{\varphi'_{\alpha k_1, \tau_1}}
=\frac{|J_{\alpha k_1 0}|}{\varphi'_{\alpha k_1, 1}}\cdot \frac{\varphi'_{\alpha k_1, 1}}{\varphi'_{\alpha k_1, \tau_1}}
\geq \frac{r_*}{C}\frac{\varphi'_{\alpha k_1, 1}}{\varphi'_{\alpha k_1, 2}}\to \infty,
$$
as $|\alpha|\to \infty$, which is equivalent to $|\sigma|\to \infty$.

 Therefore, by Lemma \ref{lem:binarytree}, we have $\dim_C E=\dim_C F=1$.
\end{proof}

Until now, we have all preparations for showing Theorem \ref{thm:simple-IFS}.

\begin{proof}[$\textbf{Proof of Theorem \ref{thm:simple-IFS} }$]
By the above Theorem \ref{thm:Lip} we have
 $\dim_C(\bar{K})=\dim_C(E)$. Then by Lemma \ref{lem:binary} the theorem is proved.
 \end{proof}

\begin{theorem}\label{thm:last}
 If $K$ is not uniformly disconnected, then $\dim_C K\geq 1$.
\end{theorem}

\begin{proof} Since $K$  is not uniformly disconnected, by Theorem \ref{thm:Lalley}, there exists a vertex $u
=(g_1,g_2,...,g_s)$ such that the fiber IFS $G(u)$ has attractor $[0,1]$.

Write the fiber IFS $G(u)$ as
$$
G(u)=\{h_0,\dots, h_{m-1}\}.
$$
Let $F=\{\phi_j=(\varphi_{j,1},\dots, \varphi_{j, d})\}_{j=0}^{m-1}$ be a subIFS of $\Phi$ such that
$$
\phi_{j}=(g_1,\dots, g_s, h_j, \varphi_{j,s+2}, \dots, \varphi_{j,d}).
$$
(The choices of $(\varphi_{j,s+2}, \dots, \varphi_{j,d})$ is not unique.)
Let
$$F_0=\{(h_j, \varphi_{j,s+2}, \dots, \varphi_{j,d})\}_{j=0}^{m-1}.$$
 Let $\bar{K}$ be the attractor
of $F_0$.
By Theorem \ref{thm:simple-IFS}, we have $\dim_C \bar{K}=1$.

Let $z_0$ be the fixed point of the map $(g_1,\dots, g_s):\R^s\to \R^s$. Then
$$
z_0\times \bar{K}\subset K.
$$
Consequently, $\dim_C K\geq \dim_C \bar{K}= 1$. The theorem is proved.
\end{proof}

\begin{proof}[\textbf{Proof of Theorem \ref{thm:dim0}}]
   If $K$ is not uniformly disconnected, then $\dim_C K\geq 1$  by Theorem \ref{thm:last}.
If $K$ is  uniformly disconnected, then $\dim_C K=0$ by
  Lemma \ref{lem:first}.
\end{proof}

\end{document}